\documentclass[11pt,reqno]{amsproc}
\allowdisplaybreaks

\usepackage{amssymb}
\usepackage{stmaryrd}

\usepackage{enumerate}
\usepackage{fullpage}
\usepackage{textcomp}
\usepackage{lettrine}

\usepackage{graphicx}
\usepackage{subfigure}
\usepackage{morefloats}

\usepackage[debug=false, colorlinks=true, pdfstartview=FitV,
linkcolor=blue, citecolor=blue, urlcolor=blue]{hyperref}
\usepackage[semicolon,square,authoryear,sort]{natbib}

\usepackage[doipre={DOI:~}]{uri}

\usepackage{color}
\usepackage{colortbl}
\usepackage[usenames,dvipsnames,table]{xcolor}
\usepackage[most]{tcolorbox}

\newlength{\drop}
\definecolor{amethyst}{rgb}{0.6, 0.4, 0.8}
\definecolor{burgundy}{rgb}{0.5, 0.0, 0.13}

\usepackage{longtable} 
\usepackage{multirow} 
\usepackage{rotating} 
\usepackage{bigstrut} 
\usepackage{hhline} 

\usepackage{amsthm}

\newtheoremstyle{remboldstyle}
  {}{}{}{}{\bfseries}{.}{.5em}{{\thmname{#1 }}{\thmnumber{#2}}{\thmnote{ (#3)}}}
\theoremstyle{remboldstyle}

\newtheorem{theorem}{Theorem}
\newtheorem{lemma}[theorem]{Lemma}
\newtheorem{proposition}[theorem]{Proposition}

\newtheorem{remark}{Remark}

\title{\textbf{Dynamic properties of double porosity/permeability model}}

\author{\textbf{Kalyana B.~Nakshatrala} \\
  {\small Associate Professor, Department of Civil and Environmental Engineering \\
  University of Houston, Houston, Texas 77204, USA.}\\
  {\small email:~\texttt{knakshatrala@uh.edu}, phone: +1-713-743-4418}}


\begin{document}

\begin{titlepage}
  \drop=0.1\textheight
  \centering
  \vspace*{\baselineskip}
  \rule{\textwidth}{1.6pt}\vspace*{-\baselineskip}\vspace*{2pt}
  \rule{\textwidth}{0.4pt}\\[\baselineskip]
       {\Large \textbf{\color{burgundy}
       Dynamic  properties of double porosity/permeability model}}\\[0.3\baselineskip]
       \rule{\textwidth}{0.4pt}\vspace*{-\baselineskip}\vspace{3.2pt}
       \rule{\textwidth}{1.6pt}\\[\baselineskip]
       \scshape
       An e-print of this paper is available on arXiv. \par
       \vspace*{1\baselineskip}
       Authored by \\[\baselineskip]

  {\Large K.~B.~Nakshatrala\par}
  {\itshape Department of Civil \& Environmental Engineering \\
  University of Houston, Houston, Texas 77204. \\
  \textbf{phone:} +1-713-743-4418, \textbf{e-mail:} knakshatrala@uh.edu \\
  \textbf{website:} http://www.cive.uh.edu/faculty/nakshatrala}\\[\baselineskip]

  \vfill
  {\scshape 2023} \\
  {\small Computational \& Applied Mechanics Laboratory} \par
\end{titlepage}

\begin{abstract}
Understanding fluid movement in multi-pored materials is vital for energy security and physiology. For instance, shale (a geological material) and bone (a biological material) exhibit multiple pore networks. Double porosity/permeability models provide a mechanics-based approach to describe hydrodynamics in aforesaid porous materials. However, current theoretical results primarily address steady-state response, and their counterparts in the transient regime are still wanting. The chief aim of this paper is to fill this knowledge gap. We present three principal properties---with rigorous mathematical arguments---that the solutions under the double porosity/permeability model satisfy in the transient regime: backward-in-time uniqueness, reciprocity, and a variational principle. We employ the ``energy method''---exploiting the physical total kinetic energy of the flowing fluid---to establish the first property and Cauchy-Riemann convolutions to prove the next two. The results reported in this paper---qualitatively describe the dynamics of fluid flow in double-pored media---have (a) theoretical significance, (b) practical applications, and (c) considerable pedagogical value. In particular, these results will benefit practitioners and computational scientists in checking the accuracy of numerical simulators. The backward-in-time uniqueness lays a firm theoretical foundation for pursuing inverse problems in which one predicts the prescribed initial conditions based on data available about the solution at a later instance. 
\end{abstract}

\maketitle

\vspace{-0.3in}

\setcounter{figure}{0}   


\section*{PRINCIPAL NOTATION}

\vspace{-0.1in}

\begin{longtable}{ll}\hline
  \multicolumn{1}{c}{\textbf{Symbol}} & \multicolumn{1}{c}{\textbf{Quantity}} \\
   \hline \multicolumn{2}{c}{\emph{Mathematical operators}} \\ \hline
  $\bullet$ & dot product (the standard inner product on Euclidean spaces) \\
  $\star$ & time convolution operation \\
  $\mathrm{div}[\cdot]$ & spatial divergence operator \\
  $\mathrm{grad}[\cdot]$ & spatial gradient operator \\
  \hline \multicolumn{2}{c}{\emph{Time-related quantities}} \\ \hline
  $t$ & time \\ 
  $\mathcal{T}$ & length of the time interval of interest \\ 
  $\partial(\cdot)/\partial t$ & partial derivative with respect to time \\
  \hline \multicolumn{2}{c}{\emph{Geometry-related quantities}} \\ \hline
  $\Omega$ & domain \\
  $\partial \Omega$ & boundary of the domain \\
  $\Gamma_1^{p}$ & part of the boundary with a prescribed macro pressure \\
  $\Gamma_2^{p}$ & part of the boundary with a prescribed micro pressure \\
  $\Gamma_1^{u}$ & part of the boundary with a prescribed normal component of macro velocity \\
  $\Gamma_2^{u}$ & part of the boundary with a prescribed normal component of micro velocity \\
  $\phi_{1}(\mathbf{x})$ & porosity of the macro pore-network [dimensionless] \\
  $\phi_{2}(\mathbf{x})$ & porosity of the micro pore-network [dimensionless]  \\
  $\widehat{\mathbf{n}}(\mathbf{x})$ & unit outward normal vector on the boundary \\
  $\mathbf{x}$ & a spatial point \\
  \hline \multicolumn{2}{c}{\emph{Time-dependent solution fields}} \\ \hline
  $p_{1}(\mathbf{x},t)$ & pressure in the macro pore-network (``macro pressure'') [$\mathrm{N/m^2}$] \\
  $p_{2}(\mathbf{x},t)$ & pressure in the micro pore-network (``micro pressure'') [$\mathrm{N/m^2}$] \\
  $\mathbf{u}_{1}(\mathbf{x},t)$ & Darcy velocity in the macro pore-network (``macro velocity'') [m/s] \\
  $\mathbf{u}_{2}(\mathbf{x},t)$ & Darcy velocity in the micro pore-network (``micro velocity'') [m/s] \\
  \hline \multicolumn{2}{c}{\emph{Prescribed quantities}} \\ \hline
  $p_{\mathrm{p}1}(\mathbf{x},t)$ & prescribed macro pressure on the boundary 
  $\Gamma_{1}^{p}$ [$\mathrm{N/m^2}$] \\
  $p_{\mathrm{p}2}(\mathbf{x},t)$ & prescribed micro pressure on the boundary $\Gamma_{2}^{p}$ [$\mathrm{N/m^2}$] \\
  $u_{n1}(\mathbf{x},t)$ & prescribed normal component of macro velocity on the boundary $\Gamma_{1}^{u}$ [$\mathrm{m/s}$]  \\
  $u_{n2}(\mathbf{x},t)$ & prescribed normal component of micro velocity on the boundary $\Gamma_{2}^{u}$ [$\mathrm{m/s}$] \\
  $\mathbf{u}_{01}(\mathbf{x})$ & prescribed initial velocity in the macro pore-network [m/s] \\
  $\mathbf{u}_{02}(\mathbf{x})$ & prescribed initial velocity in the micro pore-network [m/s] \\
  \hline \multicolumn{2}{c}{\emph{Material properties}} \\ \hline
  $\beta$ & mass transfer coefficient [dimensionless] \\
  $\mu$ & fluid's coefficient of dynamic viscosity [kg/m/s] \\
  $\gamma$ & fluid's true density [$\mathrm{kg}/\mathrm{m}^3$] \\
  $\mathbf{K}_{1}(\mathbf{x})$ & permeability tensor field of the macro pore-network [$\mathrm{m}^2$] \\
  $\mathbf{K}_{2}(\mathbf{x})$ & permeability tensor field of the micro pore-network [$\mathrm{m}^2$] \\
  \hline \multicolumn{2}{c}{\emph{Other symbols and abbreviations}} \\
  \hline
  $\mathbb{R}^{3}$ & three-dimensional Euclidean space \\
  DPP & double porosity/permeability \\
  IBVP & initial boundary value problem \\
  PDE & partial differential equation \\
  \hline
\end{longtable}

\vspace{-0.3in}

\section{INTRODUCTION AND MOTIVATION}
\label{Sec:S1_Backwards_Intro}

\lettrine[findent=2pt]{\fbox{\textbf{I}}}{t is not hyperbole to say every material is porous}---when looked at with a sufficient spatial resolution. Equally striking: materials that are porous at the continuum scale ($> 10^{-6}$ m) are also ubiquitous, ranging from biological materials (e.g., bone, tissue), geological materials (e.g., shale,  clays) to synthetic materials (e.g., surgical masks, water filters). Amongst a myriad of processes possible in porous materials, flow through porous media is prominent. This transport phenomenon enables several functionalities in living organisms (e.g., waste removal in kidneys), and it is vital to many technological pursuits (e.g., masks for protection against pollutants and pathogens, and extraction of hydrocarbons from the subsurface, to name a few). 

There is an outpouring of interest in understanding how fluids flow in porous materials comprising \emph{two or more} distinct pore networks. This gush of activity is for four primary reasons:
\begin{enumerate}
\item The success of cutting-edge initiatives, such as extracting hydrocarbons from tight shale (a double-pored geological material), depends on controlling the transport of fluids through two distinct pore structures.
\item Recent advances in imaging techniques offer an extraordinary resolution to study transport in biological materials (e.g., bone), which are inherently hierarchical, thereby facilitating potentially new therapeutic remedies.
\item New synthetic materials and devices, such as facial masks and water purification filters, use distinct but interconnected pore networks to remove impurities (e.g., bacteria, dust, particulates, and heavy metals). 
\item Lastly, additive manufacturing now allows the creation of unprecedented hierarchical porous structures.
\end{enumerate}

But fluid flow in double-pored materials is intricate. The 
classical Darcy equations---often used in studies on flow 
through porous media---cannot adequately capture such complex 
flow behaviors; Darcy's model relates the discharge velocity 
to the pressure gradient and assumes a single pore network. 
Consequently, double porosity/permeability (DPP) models have 
appeared that account for flow within different pore networks 
as well as mass transfer across them. The literature often 
attributes \citet{barenblatt1960basic} as the first work 
on DPP. After that, many researchers have developed various 
DPP models with varying complexity, which include \citep{chen1989transient,Arbogast_Douglas_Hornung_1990,Dykhuizen_v26_p351_1990_WRR,Boutin_Royer_2015,Borja_Koliji_2009,Choo_White_Borja_2016_IJG}.  

In this paper, we utilize the model put forth by \citet{nakshatrala2018modeling}, hereafter referred to as \emph{the DPP model}. The authors in the cited reference have derived the associated governing equations by appealing to the theory of interacting continua and the maximization of the rate of dissipation hypothesis---a more stringent form of the second law of thermodynamics. Unsurprisingly, the DPP model is mathematically more complicated than the Darcy equations---comprising four coupled partial differential equations in four field variables. The referenced article also provides analytical solutions for canonical problems and presents various qualitative properties that the \emph{steady-state} solutions of the DPP model satisfy: minimum dissipation theorem, maximum principles, and reciprocal theorem. 

On the temporal front, \citet{nakshatrala2021boundedness} has shown that the solutions under the DPP model are Lyapunov stable---small changes to the initial conditions give rise to nearby solution trajectories over time. Nevertheless, the literature offers limited results for the DPP model in the transient regime compared to the steady state. In a quest to fill this knowledge gap, the principal aim of this paper is to present three qualitative properties that the DPP model satisfies in the transient regime:  backward-in-time uniqueness, a reciprocal relation, and a variational principle. To bring out the import of our work, we elaborate on how such properties arise in studying other popular time-dependent mathematical models from mechanics and applied mathematics. 

 ``Backward-in-time'' solutions arise in mathematical analysis, numerical modeling (e.g., inverse problems), and practical applications (e.g., control systems). Thus, understanding these solutions' behavior (e.g., uniqueness) is essential for studying time-dependent partial differential equations (PDEs), often governed by parabolic and hyperbolic equations \citep{evans2002partial}. Analysis of backward-in-time solutions addresses the following question: if we know the solution of a PDE at a later instance, what can we infer about the behavior of the solution at earlier times? However, one should not confuse the \emph{backward-in-time solution} with the \emph{solution to the backward problem}, which too appears often in mathematical analysis. 
 
To elucidate, we briefly discuss the forward and backward problems in the study of partial differential equations. A \emph{forward problem} is an initial boundary value problem (IBVP) with a prescribed initial condition, prescribed boundary conditions, and a system of partial differential equations defined in terms of unknowns. A \emph{solution to the forward problem} provides values to the unknown variables for later instances (i.e., beyond the time at which the initial condition is prescribed), satisfying all the governing equations of the forward problem; see \textbf{Fig. ~\ref{Fig1:Backwards_solutions}}. On the other hand, the \emph{backward problem} comprises partial differential equations and boundary conditions defined for earlier times alongside a prescribed value for the solution later. Patently, a \emph{solution to the backward problem} solves the backward problem. One can obtain the governing equations of a backward problem corresponding to a forward problem by reversing the direction of time (i.e., replacing $t$ with $-t$). For second-order hyperbolic equations (such as the wave equation), the well-posedness of the forward problem manifestly implies that the backward problem is also well-posed; the structural forms of forward and backward problems are identical, as the forward problem is invariant under time reversal for wave equations. However, for parabolic equations, which is the case in this paper, the backward problem---not identical to the forward problem---is ill-posed. Solutions to the backward problem might not exist; even if a solution exists, it might not continuously depend on the prescribed input data and can blow up in a finite time \citep{pao2012nonlinear}.

Contrasting with a backward problem, the equations that govern a backward-in-time solution are still that of the forward problem. The backward-in-time \emph{uniqueness} asserts that if two solutions satisfying a forward problem match at an instance of time, then these two solutions must coincide for all prior times. Proving backward-in-time uniqueness for second-order hyperbolic equations is trivially provided by the uniqueness of the backward problem. One cannot use a similar argument for parabolic equations, as the associated backward problem is ill-posed. So, backward-in-time uniqueness for first-order transient systems and parabolic equations is a subtle and surprising property, as echoed by \citet{evans2002partial}, ``This is not at all obvious.'' 
Heat equation, a well-known first-order transient system, possesses the backward-in-time uniqueness property. However, all first-order systems are not guaranteed to have such a property. Heeding to the ever-growing interest in inverse problems concerning flow through porous media, we show the DPP model also has backward-in-time uniqueness; we avail an \emph{energy method} to establish the stated result. 

``Reciprocal relations'' enjoy a vibrant history in mechanics---especially in solid mechanics and structural analysis. In 1864, \textsc{James Clerk Maxwell} proposed a reciprocal relationship in the study of structural frames, showing that the flexibility matrix is symmetric \citep{maxwell1864calculation}. Later, in  1872, \textsc{Enrico Betti}---an Italian mathematician and the eponym of Betti's theorem, also referred to as Betti's relations---established reciprocity in \emph{elastostatics} \citep{betti1872teoria}. For a while, a common belief was that reciprocity was pertinent to linear problems until \textsc{Clifford Truesdell} showed that even (nonlinear) hyperelasticity possesses a reciprocal relation and established the equivalence between reciprocity and the existence of a stored energy functional \citep{truesdell1963meaning}.  

For time-dependent mechanics problems, \textsc{Lord Rayleigh} (originally called \textsc{John William Strutt}) posed a reciprocal relation for vibrations \citep{strutt1873some}. \textsc{Dario Graffi} gave the modern version of reciprocal relations for \emph{elastodynamics} \citep{graffi1946sul}. \textsc{Adrianus Teunis de Hoop} extended such relations to linear viscoelastic models \citep{de1966elastodynamic}; also refer to his handbook \citep{de1995handbook}.  A comprehensive discussion of reciprocity in elastic solids can be found in \citep{gurtin1973linear,achenbach2003reciprocity,achenbach2006reciprocity}. But reciprocity is a unique property and should not be taken for granted. Non-reciprocity can occur even in elastodynamics, for example, see \citep{blanchard2018non}.

Besides academic significance, Betti's theorem has several practical utilities: it enables solving seemingly complicated practical problems involving concentrated loads, contact problems, and punch problems \citep{mossakovskii1953application,moore2020extending}. Researchers have also utilized reciprocal relations under the boundary element method to obtain numerical solutions \citep{aliabadi2020boundary,beskos1987boundary,panagiotopoulos2011three}. Beyond solid mechanics, \citet{shabouei2016mechanics} have extended reciprocity to flow through porous media problems; they developed Betti-type reciprocal relations for Darcy and Darcy-Brinkman equations and used those relations as \emph{a posteriori} measures to assess the accuracy of numerical solutions. Given the theoretical and practical import of reciprocity, it is desirable to have a similar result under the dynamic DPP model. Ergo, we establish one such by availing \emph{Cauchy-Riemann convolutions}.

``Variational principles'' have been undoubtedly instrumental in advancing mechanics and physics \citep{lanczos1986variational}. For instance, in solid mechanics, the principles of minimum potential energy and minimum complementary energy have been central to the progress of elasticity and plasticity \citep{washizu1968variational}, early developments in the finite element method \citep{Shames_Dym}, and efforts to bound material constants \citep{huet1990application}. \textsc{Warner Koiter} aptly used the energy method---a variational technique---to advance the field of elastic stability \citep{koiter1967stability}. Variational principles are not just limited to quasi-static or steady-state problems. Two famous variational approaches---the Lagrangian and the Hamiltonian formalisms---led to significant developments in studying time-dependent phenomena in areas such as classical mechanics \citep{deriglazov2016classical}, dynamics \citep{rosenberg1977analytical}, and quantum mechanics \citep{griffiths2018introduction}. 

However, the mentioned formalisms do not apply to first-order transient systems. Noteworthy, there was a time when it was believed that variational principles could not exist for mathematical models involving first-order (or any odd-numbered order) time derivatives, as the resulting governing equations would not be self-adjoint \citep{washizu1968variational}. Surprisingly, \textsc{Morton Gurtin} showed otherwise; he constructed a variational principle for the classical heat equation (a first-order time-dependent mathematical model described using a scalar field variable) using convolutions \citep{gurtin1964variational}. Noting the standing of variational principles in mechanics, we pose a variational principle for the DPP model. We employ the technique put forth by \textsc{Gurtin}; however, we make several modifications, as our system comprises a coupled system of PDEs with four independent field variables and accounts for fluid's incompressibility---an internal constraint. 

The three results presented in this paper will further elevate the double porosity/permeability (DPP) model's status in describing fluid flow in porous media exhibiting distinct pore networks. The plan for the rest of this paper is as follows. We first outline the governing equations: an initial boundary value problem arising from the DPP model (\S\ref{Sec:S2_Backwards_GE}). We then prove the backward-in-time uniqueness property (\S\ref{Sec:S3_Backwards_Uniqueness}), followed by a reciprocal theorem (\S\ref{Sec:S4_Backwards_Reciprocity}). Next, we present a variational principle for the DPP model (\S\ref{Sec:S5_Backwards_Variational}). Finally, we draw conclusions alongside a brief discussion of possible future works (\S\ref{Sec:S6_Backwards_Closure}).

\begin{figure}
    \includegraphics[scale=2]{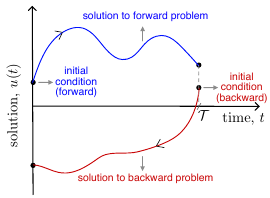}
    \caption{Solutions to the forward and backward problems. The \emph{backward-in-time solution}, one of the central topics of this paper, is not necessarily a solution to the backward problem. \label{Fig1:Backwards_solutions}}
\end{figure}

\section{TRANSIENT DOUBLE POROSITY/PERMEABILITY (DPP) MODEL}
\label{Sec:S2_Backwards_GE}

Consider a porous domain $\Omega$ comprising two distinct pore networks. The porous skeleton is rigid; thus, the deformation of the porous solid is negligible. An incompressible fluid flows through the porous medium with seepage in both pore networks and possibly exchange mass between the networks. Henceforth, we refer to the two networks as ``macro'' and ``micro.'' To distinguish easily, we use the subscripts ``1'' and ``2'' on the quantities related to macro- and micro-networks, respectively. 

For technical reasons, we assume the domain to be an open set---bounded by a boundary $\partial \Omega = \overline{\Omega} \setminus \Omega$, in which the overline represents the set closure \citep{evans2002partial}. $\mathbf{x} \in \overline{\Omega}$ denotes a spatial point, and $t \in [0,\mathcal{T}]$ symbolizes the time, with $\mathcal{T} < +\infty$ representing the length of the time interval of interest. $\partial(\cdot)/\partial t$ is the partial time derivative, while $\mathrm{grad}[\cdot]$ and $\mathrm{div}[\cdot]$ stand for the spatial gradient and divergence operators, respectively. $\widehat{\mathbf{n}}(\mathbf{x})$ depicts the unit outward normal to the boundary.

For the fluid, $\gamma$ and $\mu$ denote the true density and the coefficient of dynamic viscosity, respectively. $\beta$ represents the coefficient of mass transfer between the pore networks---a dimensionless characteristic of a double-pored medium. Each pore network has inherent hydrodynamical properties governing fluid flow through it. $\phi_1(\mathbf{x})$ and $\phi_2(\mathbf{x})$ denote the porosity fields for the macro- and micro-networks, respectively;  since the porous solid is rigid, the porosity fields are independent of time. $\mathbf{K}_1(\mathbf{x})$ and $\mathbf{K}_2(\mathbf{x})$ represent the permeability fields for the macro- and micro-networks, respectively. The permeabilities are symmetric and positive-definite tensor fields---and hence invertible. $\mathbf{u}_1(\mathbf{x},t)$ and $\mathbf{u}_2(\mathbf{x},t)$ are the Darcy velocities of the fluid---referred to as the \emph{macro-velocity} and the \emph{micro-velocity}, respectively. $p_1(\mathbf{x},t)$ and $p_2(\mathbf{x},t)$ are the fluid's pressures in the macro- and micro-networks: the \emph{macro-pressure} and the \emph{micro-pressure}, respectively.

For the macro-network, the boundary is divided into two complementary parts: $\Gamma_{1}^{u}$ and $\Gamma_{1}^{p}$. $\Gamma_{1}^{u}$ is part of the boundary on which the normal component of the macro-velocity, $u_{n1}(\mathbf{x},t)$, is prescribed. $\Gamma_{1}^{p}$ denotes that part of the boundary on which the macro-pressure, $p_{\mathrm{p}1}(\mathbf{x},t)$, is prescribed. For mathematical well-posedness, we require the following complementary conditions to hold for these boundary partitions:
\begin{align}
    \label{Eqn:Backwards_macro_gamma_partition}
    \Gamma_{1}^{u} \cup \Gamma_{1}^{p} = \partial \Omega 
    \quad \mathrm{and} \quad 
    \Gamma_{1}^{u} \cap \Gamma_{1}^{p} = \emptyset
\end{align}
In like manner, $\Gamma_{2}^{u}$ and $\Gamma_{2}^{p}$ denote the parts of the boundary on which the normal component of the micro-velocity, $u_{n2}(\mathbf{x},t)$, and micro-pressure, $p_{\mathrm{p}2}(\mathbf{x},t)$, are prescribed, respectively. Correspondingly, we require:  
\begin{align}
    \label{Eqn:Backwards_micro_gamma_partition}
    \Gamma_{2}^{u} \cup \Gamma_{2}^{p} = \partial \Omega 
    \quad \mathrm{and} \quad 
    \Gamma_{2}^{u} \cap \Gamma_{2}^{p} = \emptyset
\end{align}

The governing equations describing the flow of an incompressible fluid in a double-pored medium take the following form: 
\begin{subequations}
\label{Eqn:Backwards_DPP_GE}
\begin{alignat}{2}
    \label{Eqn:Backwards_BoLM_1}
    &\frac{\gamma}{\phi_1(\mathbf{x})}  \frac{\partial \mathbf{u}_{1}(\mathbf{x},t)}{\partial t} 
    +  \mu \, \mathbf{K}_{1}^{-1}(\mathbf{x}) \, \mathbf{u}_{1}(\mathbf{x},t) + \mathrm{grad}\big[p_1(\mathbf{x},t)\big] = \gamma \, \mathbf{b}_1(\mathbf{x},t) 
    &&  \qquad \mathrm{in} \; \; \Omega \times (0,\mathcal{T}] \\
    \label{Eqn:Backwards_BoLM_2}
    &\frac{\gamma}{\phi_2(\mathbf{x})}  \frac{\partial \mathbf{u}_{2}(\mathbf{x},t)}{\partial t} 
    +  \mu \, \mathbf{K}_{2}^{-1}(\mathbf{x}) \, \mathbf{u}_{2}(\mathbf{x},t) + \mathrm{grad}\big[p_2(\mathbf{x},t)\big] = \gamma \, \mathbf{b}_2(\mathbf{x},t) 
    &&  \qquad \mathrm{in} \; \; \Omega \times (0,\mathcal{T}] \\
    \label{Eqn:Backwards_BoM_1}
    &\mathrm{div}\big[\mathbf{u}_{1}(\mathbf{x},t)\big] 
    = -\frac{\beta}{\mu} \big(p_1(\mathbf{x},t) - p_2(\mathbf{x},t)\big)
    &&  \qquad \mathrm{in} \; \; \Omega \times (0,\mathcal{T}] \\
    \label{Eqn:Backwards_BoM_2}
    &\mathrm{div}\big[\mathbf{u}_{2}(\mathbf{x},t)\big] 
    = +\frac{\beta}{\mu} \big(p_1(\mathbf{x},t) - p_2(\mathbf{x},t)\big)
    &&  \qquad \mathrm{in} \; \; \Omega \times (0,\mathcal{T}] \\
    \label{Eqn:Backwards_vBC_1}
    &\mathbf{u}_1(\mathbf{x},t) \bullet \widehat{\mathbf{n}}(\mathbf{x}) = u_{n1}(\mathbf{x},t) 
    && \qquad \mathrm{on} \; \; \Gamma^{u}_{1} \times [0,\mathcal{T}] \\
    \label{Eqn:Backwards_vBC_2}
    &\mathbf{u}_2(\mathbf{x},t) \bullet \widehat{\mathbf{n}}(\mathbf{x}) = u_{n2}(\mathbf{x},t) 
    && \qquad \mathrm{on} \; \; \Gamma^{u}_{2} \times [0,\mathcal{T}] \\
    \label{Eqn:Backwards_pBC_1}
    &p_1(\mathbf{x},t) = p_{\mathrm{p}1}(\mathbf{x},t) 
    && \qquad \mathrm{on} \; \; \Gamma^{p}_{1} \times [0,\mathcal{T}] \\
    \label{Eqn:Backwards_pBC_2}
    &p_2(\mathbf{x},t) = p_{\mathrm{p}2}(\mathbf{x},t) 
    && \qquad \mathrm{on} \; \; \Gamma^{p}_{2} \times [0,\mathcal{T}] \\
    \label{Eqn:Backwards_vIC_1}
    &\mathbf{u}_1(\mathbf{x},t=0) = \mathbf{u}_{01}(\mathbf{x}) 
    && \qquad \mathrm{in} \; \; \Omega \\
    \label{Eqn:Backwards_vIC_2}
    &\mathbf{u}_2(\mathbf{x},t=0) = \mathbf{u}_{02}(\mathbf{x}) 
    && \qquad \mathrm{in} \; \; \Omega 
\end{alignat}
\end{subequations}
where $\bullet$ depicts the dot product between two vectors, $\mathbf{u}_{01}(\mathbf{x})$ and $\mathbf{u}_{01}(\mathbf{x})$ denote the prescribed initial velocities, and $\mathbf{b}_{1}(\mathbf{x},t)$ and $\mathbf{b}_{2}(\mathbf{x},t)$ are the specific body forces appertaining to the macro- and micro-networks, respectively. To clarify, $\gamma \, \mathbf{b}_1(\mathbf{x},t)$ is the ``apparent'' body force density---the quotient of the body force per unit volume of the porous medium---acting on the fluid in the macro-network. A similar interpretation holds for $\gamma \, \mathbf{b}_2(\mathbf{x},t)$.

Equations \eqref{Eqn:Backwards_BoLM_1} and \eqref{Eqn:Backwards_BoLM_2} represent the balance of linear momentum for the macro- and micro-networks, respectively. On the other hand, Eqs.~\eqref{Eqn:Backwards_BoM_1} and \eqref{Eqn:Backwards_BoM_2} depict the mass balance for the two networks, accounting for the inter-network mass exchange. We direct the reader to \citep{nakshatrala2018modeling} for a systematic mathematical derivation of the above governing equations.

\subsection{Mathematical preliminaries}
The principal results presented in forthcoming sections avail the properties of ``convolutions.'' Thus, we record a few concomitant definitions and properties that we resort to later in this paper. 

Consider two time-dependent scalar fields: $f(\mathbf{x},t)$ and $g(\mathbf{x},t)$, and two time-dependent vector fields: $\mathbf{f}(\mathbf{x},t)$ and $\mathbf{g}(\mathbf{x},t)$. For these fields, one can define the following three forms of Cauchy-Riemann convolutions: 
\begin{align}
    \label{Eqn:Backwards_CR_type_1}
    &[f \star g](\mathbf{x},t)
    := \int_{0}^{t} f(\mathbf{x},t  - \tau) 
    \, g(\mathbf{x},\tau) \, \mathrm{d} \tau \\
    \label{Eqn:Backwards_CR_type_2}
    &[f \star \mathbf{f}](\mathbf{x},t)
    := \int_{0}^{t} f(\mathbf{x},t  - \tau) 
    \, \mathbf{f}(\mathbf{x},\tau) \, \mathrm{d} \tau \\
    \label{Eqn:Backwards_CR_type_3}
    &[\mathbf{f} \star \mathbf{g}](\mathbf{x},t)
    := \int_{0}^{t} \mathbf{f}(\mathbf{x},t  - \tau) 
    \bullet \mathbf{g}(\mathbf{x},\tau) \, \mathrm{d} 
    \tau 
    = \int_{0}^{t} 
    f_{i}(\mathbf{x},t - \tau) \, g_{i}(\mathbf{x},\tau) \; \mathrm{d} \tau 
\end{align}
in which $f_i$ denotes the $i$-th component of the vector field $\mathbf{f}$ and, likewise, for $g_i$. Also, in writing the last expression, we have invoked Einstein's summation notation: a repeated index implies summation on the index. The first form \eqref{Eqn:Backwards_CR_type_1} is a convolution between two scalar fields, the second form \eqref{Eqn:Backwards_CR_type_2} between a scalar and a vector fields, and the third form \eqref{Eqn:Backwards_CR_type_3} between two vector fields. Note that the convolution between two vector fields is a scalar field (viz. Eq.~\eqref{Eqn:Backwards_CR_type_3}). Whenever there is no confusion, we drop the arguments and square brackets and denote the above convolutions simply by $f \star g$, $f \star \mathbf{f}$, and $\mathbf{f} \star \mathbf{g}$. 
 
We also utilize the following particular case of the second form \eqref{Eqn:Backwards_CR_type_2}:
\begin{align}
    &1 \star \mathbf{f} = \int_{0}^{t} 
    \mathbf{f}(\mathbf{x},\tau) \, \mathrm{d}\tau
\end{align}
Using Leibniz integral rule (e.g., see \citep{flanders1973differentiation}), it is easy to check that  
\begin{align}
    \label{Eqn:Backwards_convolutions_Leibniz_rule}
    \frac{\partial (1 \star \mathbf{f})}{\partial t} 
    =  \mathbf{f}(\mathbf{x},t) 
\end{align}

Cauchy-Riemann convolutions satisfy the following algebraic properties: 
\begin{alignat}{2}
    \label{Eqn:Backwards_commutative}
    &\mathbf{f} \star \mathbf{g} 
    = \mathbf{g} \star \mathbf{f} 
    &&\quad \mbox{[commutative property]} \\
    \label{Eqn:Backwards_associative}
    &(\mathbf{f} \star \mathbf{g}) \star \mathbf{h} 
    = \mathbf{f} \star (\mathbf{g} \star \mathbf{h}) 
    &&\quad \mbox{[associative property]} 
\end{alignat}
in which $\mathbf{h}(\mathbf{x},t)$ is a time-dependent vector field. For further details, see \citep{mikusinski2014operational}.

This article also uses the following identities concerning time derivatives applied to convolutions:
\begin{align}
    \label{Eqn:Backwards_f_star_dgdt}
    & \mathbf{f} \star \frac{\partial \mathbf{g}}{\partial t}
    - \frac{\partial \mathbf{f}}{\partial t} \star \mathbf{g} 
    = \mathbf{f}(\mathbf{x},0) \bullet \mathbf{g}(\mathbf{x},t) 
    - \mathbf{f}(\mathbf{x},t) \bullet \mathbf{g}(\mathbf{x},0)
    \\
    \label{Eqn:Backwards_df_star_g_dt}
    &\frac{\partial (\mathbf{f} \star \mathbf{g})}{\partial t}
    = \mathbf{f}(\mathbf{x},0) \bullet \mathbf{g}(\mathbf{x},t) + \frac{\partial \mathbf{f}}{\partial t} \star \mathbf{g} 
    = \mathbf{f}(\mathbf{x},t) \bullet \mathbf{g}(\mathbf{x},0) + \mathbf{f} \star  \frac{\partial \mathbf{g}}{\partial t}
\end{align}
We have the following (spatial) product rule involving convolutions: 
\begin{align}
\label{Eqn:Backwards_CR_spatial_product_rule}
\mathbf{f} \star \mathrm{grad}[f] = \mathrm{div}[\mathbf{f} \star f] - \mathrm{div}[\mathbf{f}] \star f 
\end{align}
In indicial notation, the above identity takes the following expanded form:
\begin{align}
\int_{0}^{t} 
f_{i}(\mathbf{x},t-\tau) \, \frac{\partial f(\mathbf{x},\tau)}{\partial x_i} \, \mathrm{d} \tau 
= \int_{0}^{t} \frac{\partial \big(f_{i}(\mathbf{x},t 
- \tau) \, f(\mathbf{x},\tau)\big)}{\partial x_i} \, \mathrm{d} \tau 
- \int_{0}^{t} \frac{\partial f_{i}(\mathbf{x},t - \tau)}{\partial x_i} \, f(\mathbf{x},\tau) \, 
\mathrm{d} \tau 
\end{align}
in which $x_i$ denotes the $i$-th component of the spatial vector $\mathbf{x}$. We have again availed Einstein's summation notation in writing the above equation. Mathematical identity \eqref{Eqn:Backwards_CR_spatial_product_rule} implies the following integral relationship as a consequence of the divergence theorem:
\begin{align}
\label{Eqn:Backwards_Greens_identity}
\int_{\Omega} \mathbf{f} \star \mathrm{grad}[f] \; \mathrm{d} \Omega 
= \int_{\partial \Omega} \Big(\mathbf{f} \bullet \widehat{\mathbf{n}}(\mathbf{x})\Big) \star f \; \mathrm{d} \Gamma 
- \int_{\Omega} \mathrm{div}[\mathbf{f}] \star f 
\; \mathrm{d} \Omega 
\end{align}

\section{BACKWARD-IN-TIME UNIQUENESS}
\label{Sec:S3_Backwards_Uniqueness}

We pose this section's central question as follows. Consider two solutions that satisfy the governing equations except for the initial velocity conditions: the two solutions meet Eqs.~\eqref{Eqn:Backwards_BoLM_1}--\eqref{Eqn:Backwards_pBC_2}; however, nothing is said about conforming to the initial velocity conditions \eqref{Eqn:Backwards_vIC_1} and \eqref{Eqn:Backwards_vIC_2}. If the two sets of velocities match (i.e., the macro-velocity of the first solution coincides with that of the second, and likewise, for the micro-velocity) at some instance of time, could the two solutions differ at prior times? The ``backward-in-time uniqueness'' theorem---proven below---asserts that the answer to the posed question is \emph{negative}.

\begin{theorem}[Backward-in-time uniqueness] 
\label{Thm:Backwards_Backward_uniqueness_theorem}
Let 
\[
\Big\{\mathbf{u}_1^{(1)}(\mathbf{x},t),\mathbf{u}_2^{(1)}(\mathbf{x},t),p_1^{(1)}(\mathbf{x},t),p_2^{(1)}(\mathbf{x},t) \Big\}
\quad \mathrm{and} \quad
\Big\{\mathbf{u}_1^{(2)}(\mathbf{x},t),\mathbf{u}_2^{(2)}(\mathbf{x},t),p_1^{(2)}(\mathbf{x},t),p_2^{(2)}(\mathbf{x},t) \Big\}
\] 
denote two solutions that satisfy Eqs.~\eqref{Eqn:Backwards_BoLM_1}--\eqref{Eqn:Backwards_pBC_2} and not necessarily the initial velocity conditions \eqref{Eqn:Backwards_vIC_1}--\eqref{Eqn:Backwards_vIC_2}.
If, at some instance of time $\mathcal{T} > 0$, we have 
\begin{align}
\label{Eqn:Backwards_theorem_condition_1}
\mathbf{u}_{1}^{(1)}(\mathbf{x},\mathcal{T}) = \mathbf{u}_{1}^{(2)}(\mathbf{x},\mathcal{T}) 
\quad \mathrm{and} \quad 
\mathbf{u}_{2}^{(1)}(\mathbf{x},\mathcal{T}) = \mathbf{u}_{2}^{(2)}(\mathbf{x},\mathcal{T}) 
\quad \forall \mathbf{x} \in \Omega 
\end{align} 
then the velocities under the two solutions coincide at all prior times: 
\begin{align}
\mathbf{u}_{1}^{(1)}(\mathbf{x},t) = \mathbf{u}_{1}^{(2)}(\mathbf{x},t) 
\quad \mathrm{and} \quad \mathbf{u}_{2}^{(1)}(\mathbf{x},t) = \mathbf{u}_{2}^{(2)}(\mathbf{x},t) 
\quad \forall\mathbf{x}\in\Omega, \forall t \in [0,\mathcal{T})
\end{align}
Moreover, the macro- and micro-pressures under the two solutions match, albeit up to an arbitrary constant: 
\begin{align}
    p_1^{(1)}(\mathbf{x},t) = p_1^{(2)}(\mathbf{x},t) + C  
    \quad \mathrm{and} \quad 
    p_2^{(1)}(\mathbf{x},t) = p_2^{(2)}(\mathbf{x},t) + C 
    \quad \forall \mathbf{x} \in \Omega, \forall t \in [0,\mathcal{T})
\end{align}
in which $C$ is a constant. Further, if either $\Gamma_1^{p} \neq \emptyset$ or $\Gamma_2^{p} \neq \emptyset$, then $C = 0$.
\end{theorem}

We must emphasize two aspects in the above theorem's statement. \emph{First}, the macro- and micro-pressures need not match the corresponding ones across the two solutions at $t = \mathcal{T}$, unlike the requirement for the velocities (i.e., Eq.~\eqref{Eqn:Backwards_theorem_condition_1}). \emph{Second}, the theorem asserts that the solution fields match under two solutions at \emph{all} times before $t = \mathcal{T}$, including the initial velocity conditions. 

We introduce suitable notation and establish a few preparatory propositions before proving the above theorem. A superposed dot denotes the derivative with respect to time. Needless to say, two superposed dots on a quantity imply the time derivative applied twice to the quantity. We also define the following differences between the respective components of the two solutions: 
\begin{subequations}
\begin{align}
    \mathbf{w}_{1}(\mathbf{x},t) 
    &:= \mathbf{u}_{1}^{(1)}(\mathbf{x},t) - 
    \mathbf{u}_{1}^{(2)}(\mathbf{x},t) \\
    \mathbf{w}_{2}(\mathbf{x},t) 
    &:= \mathbf{u}_{2}^{(1)}(\mathbf{x},t) - 
    \mathbf{u}_{2}^{(2)}(\mathbf{x},t) \\
    q_{1}(\mathbf{x},t) 
    &:= p_{1}^{(1)}(\mathbf{x},t) - 
    p_{1}^{(2)}(\mathbf{x},t) \\
    q_{2}(\mathbf{x},t) 
    &:= p_{2}^{(1)}(\mathbf{x},t) - 
    p_{2}^{(2)}(\mathbf{x},t) 
\end{align}
\end{subequations}

Clearly, the quantities $\big\{\mathbf{w}_1(\mathbf{x},t),\mathbf{w}_2(\mathbf{x},t),q_1(\mathbf{x},t),q_2(\mathbf{x},t)\big\}$ satisfy the following boundary value problem (cf. Eqs.~\eqref{Eqn:Backwards_BoLM_1}--\eqref{Eqn:Backwards_pBC_2}): 
\begin{subequations}
\begin{alignat}{2}
    \label{Eqn:Backwards_wq_BoLM_1}
    &\frac{\gamma}{\phi_1(\mathbf{x})}  \frac{\partial \mathbf{w}_{1}(\mathbf{x},t)}{\partial t} 
    +  \mu \, \mathbf{K}_{1}^{-1}(\mathbf{x}) \, \mathbf{w}_{1}(\mathbf{x},t) + \mathrm{grad}\big[q_1(\mathbf{x},t)\big] = \mathbf{0}
    &&  \qquad \mathrm{in} \; \; \Omega \times (0,\mathcal{T}] \\
    \label{Eqn:Backwards_wq_BoLM_2}
    &\frac{\gamma}{\phi_2(\mathbf{x})}  \frac{\partial \mathbf{w}_{2}(\mathbf{x},t)}{\partial t} 
    +  \mu \, \mathbf{K}_{2}^{-1}(\mathbf{x}) \, \mathbf{w}_{2}(\mathbf{x},t) + \mathrm{grad}\big[q_2(\mathbf{x},t)\big] = \mathbf{0}
    &&  \qquad \mathrm{in} \; \; \Omega \times (0,\mathcal{T}] \\
    \label{Eqn:Backwards_wq_BoM_1}
    &\mathrm{div}\big[\mathbf{w}_{1}(\mathbf{x},t)\big] 
    = -\frac{\beta}{\mu} \big(q_1(\mathbf{x},t) - q_2(\mathbf{x},t)\big)
    &&  \qquad \mathrm{in} \; \; \Omega \times (0,\mathcal{T}] \\
    \label{Eqn:Backwards_wq_BoM_2}
    &\mathrm{div}\big[\mathbf{w}_{2}(\mathbf{x},t)\big] 
    = +\frac{\beta}{\mu} \big(q_1(\mathbf{x},t) - q_2(\mathbf{x},t)\big)
    &&  \qquad \mathrm{in} \; \; \Omega \times (0,\mathcal{T}] \\
    \label{Eqn:Backwards_wq_vBC_1}
    &\mathbf{w}_1(\mathbf{x},t) \bullet \widehat{\mathbf{n}}(\mathbf{x}) = 0
    && \qquad \mathrm{on} \; \; \Gamma^{u}_{1} \times [0,\mathcal{T}] \\
    \label{Eqn:Backwards_wq_vBC_2}
    &\mathbf{w}_2(\mathbf{x},t) \bullet \widehat{\mathbf{n}}(\mathbf{x}) = 0 
    && \qquad \mathrm{on} \; \; \Gamma^{u}_{2} \times [0,\mathcal{T}] \\
    \label{Eqn:Backwards_wq_pBC_1}
    &q_1(\mathbf{x},t) = 0
    && \qquad \mathrm{on} \; \; \Gamma^{p}_{1} \times [0,\mathcal{T}] \\
    \label{Eqn:Backwards_wq_pBC_2}
    &q_2(\mathbf{x},t) = 0
    && \qquad \mathrm{on} \; \; \Gamma^{p}_{2} \times [0,\mathcal{T}] 
\end{alignat}
\end{subequations}
Notice that, in the above equations, we did not assert any initial conditions on the differences $\mathbf{w}_1(\mathbf{x},t)$ and $\mathbf{w}_2(\mathbf{x},t)$; stated differently, we did not write conditions similar to Eqs.~\eqref{Eqn:Backwards_vIC_1} and \eqref{Eqn:Backwards_vIC_2}. 

The following function plays a major role in establishing Theorem \ref{Thm:Backwards_Backward_uniqueness_theorem}: 
\begin{align}
    \label{Eqn:Backwards_E_of_t}
    \mathcal{E}(t) 
    := \int_{\Omega} \frac{1}{2} \frac{\gamma}{\phi_1(\mathbf{x})} \mathbf{w}_{1}(\mathbf{x},t) \bullet \mathbf{w}_{1}(\mathbf{x},t) \; \mathrm{d} \Omega 
    + \int_{\Omega} \frac{1}{2} \frac{\gamma}{\phi_2(\mathbf{x})} \mathbf{w}_{2}(\mathbf{x},t) \bullet \mathbf{w}_{2}(\mathbf{x},t) \; \mathrm{d} \Omega 
\end{align}
Since $\gamma$, $\phi_1(\mathbf{x})$, and $\phi_2(\mathbf{x})$ are positive, $\mathcal{E}(t)$ is non-negative. Further, if either $\mathbf{w}_1(\mathbf{x},t)$ or $\mathbf{w}_2(\mathbf{x},t)$ is non-zero, then $\mathcal{E}(t) > 0$. See \S\ref{Subsec:Backwards_Remark_on_E_of_t} below, wherein we remarked about the physical interpretation of $\mathcal{E}(t)$.

\begin{proposition}
\label{Prop:Backward_Time_deriatives_of_E}
The first and second time derivatives of $\mathcal{E}(t)$, defined in Eq.~\eqref{Eqn:Backwards_E_of_t}, are:
\begin{align}
    \label{Eqn:Backwards_dE_by_dt}
    \dot{\mathcal{E}}(t) 
    &= \int_{\Omega} \frac{\gamma}{\phi_1(\mathbf{x})} \frac{\partial \mathbf{w}_{1}(\mathbf{x},t)}{\partial t} \bullet \mathbf{w}_{1}(\mathbf{x},t) \; \mathrm{d} \Omega
    + \int_{\Omega} \frac{\gamma}{\phi_2(\mathbf{x})} \frac{\partial \mathbf{w}_{2}(\mathbf{x},t)}{\partial t} \bullet \mathbf{w}_{2}(\mathbf{x},t) \; \mathrm{d} \Omega \\
    \label{Eqn:Backwards_d2E_by_dt2}
    \ddot{\mathcal{E}}(t) 
    &= 2 \int_{\Omega} \frac{\gamma}{\phi_1(\mathbf{x})} \frac{\partial \mathbf{w}_{1}(\mathbf{x},t)}{\partial t} \bullet \frac{\partial \mathbf{w}_{1}(\mathbf{x},t)}{\partial t} \; \mathrm{d} \Omega 
    + 2 \int_{\Omega} \frac{\gamma}{\phi_2(\mathbf{x})} \frac{\partial \mathbf{w}_{2}(\mathbf{x},t)}{\partial t} \bullet \frac{\partial \mathbf{w}_{2}(\mathbf{x},t)}{\partial t} \; \mathrm{d} \Omega 
\end{align}
\end{proposition}
\begin{proof}
Since the domain does not deform (i.e., $\Omega$ is independent of time), we take the time derivative inside the integrals of Eq.~\eqref{Eqn:Backwards_E_of_t}. Noting that $\gamma$, $\phi_1(\mathbf{x})$, and $\phi_2(\mathbf{x})$ are also independent of time, the product rule for differentiation establishes the first identity:
\begin{align}
    \dot{\mathcal{E}}(t) 
    &= \int_{\Omega} \frac{\gamma}{\phi_1(\mathbf{x})} \frac{\partial \mathbf{w}_{1}(\mathbf{x},t)}{\partial t} \bullet \mathbf{w}_{1}(\mathbf{x},t) \; \mathrm{d} \Omega
    + \int_{\Omega} \frac{\gamma}{\phi_2(\mathbf{x})} \frac{\partial \mathbf{w}_{2}(\mathbf{x},t)}{\partial t} \bullet \mathbf{w}_{2}(\mathbf{x},t) \; \mathrm{d} \Omega 
\end{align}

To establish the second identity, we proceed as follows. Substituting Eqs.~\eqref{Eqn:Backwards_wq_BoLM_1} and \eqref{Eqn:Backwards_wq_BoLM_2} into the above identity, we get: 
\begin{align}
    \dot{\mathcal{E}}(t) 
    &= -\int_{\Omega} 
    \Big\{
    \mu \, \mathbf{K}_1^{-1}(\mathbf{x}) \, \mathbf{w}_1(\mathbf{x},t) + 
    \mathrm{grad}\big[q_1(\mathbf{x},t)\big]
    \Big\} 
    \bullet \mathbf{w}_{1}(\mathbf{x},t) \; \mathrm{d} \Omega \nonumber \\
    &\qquad \qquad -\int_{\Omega} 
    \Big\{
    \mu \, \mathbf{K}_2^{-1}(\mathbf{x}) \, \mathbf{w}_2(\mathbf{x},t) + 
    \mathrm{grad}\big[q_2(\mathbf{x},t)\big]
    \Big\} 
    \bullet \mathbf{w}_{2}(\mathbf{x},t) \; \mathrm{d} \Omega
\end{align}
Applying Green's identity on the two ``grad'' terms, noting the boundary partitions \eqref{Eqn:Backwards_macro_gamma_partition} and \eqref{Eqn:Backwards_micro_gamma_partition}, and rearranging the terms, we obtain:
\begin{align}
    \dot{\mathcal{E}}(t) 
    &= -\int_{\Omega} 
    \mu \, \mathbf{K}_1^{-1}(\mathbf{x}) \, \mathbf{w}_1(\mathbf{x},t) 
    \bullet \mathbf{w}_{1}(\mathbf{x},t) \; \mathrm{d} \Omega 
    -\int_{\Omega} 
    \mu \, \mathbf{K}_2^{-1}(\mathbf{x}) \, \mathbf{w}_2(\mathbf{x},t) 
    \bullet \mathbf{w}_{2}(\mathbf{x},t) \; \mathrm{d} \Omega \nonumber \\ 
    &\qquad +\int_{\Omega} 
    \mathrm{div}\big[\mathbf{w}_1(\mathbf{x},t)\big] \, q_{1}(\mathbf{x},t) \; \mathrm{d} \Omega 
    +\int_{\Omega} 
    \mathrm{div}\big[\mathbf{w}_2(\mathbf{x},t)\big] \, q_{2}(\mathbf{x},t) \; \mathrm{d} \Omega
    \nonumber \\ 
    &\qquad -\int_{\Gamma_1^{u}} 
    q_{1}(\mathbf{x},t) \, 
    \Big(\mathbf{w}_1(\mathbf{x},t) \bullet 
    \widehat{\mathbf{n}}(\mathbf{x}) \Big) 
    \; \mathrm{d} \Gamma  
    -\int_{\Gamma_1^{p}} 
    q_{1}(\mathbf{x},t) \, 
    \Big(\mathbf{w}_1(\mathbf{x},t) \bullet 
    \widehat{\mathbf{n}}(\mathbf{x}) \Big) 
    \; \mathrm{d} \Gamma  
    \nonumber \\ 
    &\qquad -\int_{\Gamma_2^{u}} 
    q_{2}(\mathbf{x},t) \, 
    \Big(\mathbf{w}_2(\mathbf{x},t) \bullet 
    \widehat{\mathbf{n}}(\mathbf{x}) \Big) 
    \; \mathrm{d} \Gamma  
    -\int_{\Gamma_2^{p}} 
    q_{2}(\mathbf{x},t) \, 
    \Big(\mathbf{w}_2(\mathbf{x},t) \bullet 
    \widehat{\mathbf{n}}(\mathbf{x}) \Big) 
    \; \mathrm{d} \Gamma  
\end{align}
Boundary conditions \eqref{Eqn:Backwards_wq_vBC_1}--\eqref{Eqn:Backwards_wq_pBC_2} imply that the last four integrals vanish, and hence we have:  
\begin{align}
    \dot{\mathcal{E}}(t) 
    &= -\int_{\Omega} 
    \mu \, \mathbf{K}_1^{-1}(\mathbf{x}) \, \mathbf{w}_1(\mathbf{x},t) 
    \bullet \mathbf{w}_{1}(\mathbf{x},t) \; \mathrm{d} \Omega 
    -\int_{\Omega} 
    \mu \, \mathbf{K}_2^{-1}(\mathbf{x}) \, \mathbf{w}_2(\mathbf{x},t) 
    \bullet \mathbf{w}_{2}(\mathbf{x},t) \; \mathrm{d} \Omega \nonumber \\ 
    &\qquad +\int_{\Omega} 
    \mathrm{div}\big[\mathbf{w}_1(\mathbf{x},t)\big] \, q_{1}(\mathbf{x},t) \; \mathrm{d} \Omega 
    +\int_{\Omega} 
    \mathrm{div}\big[\mathbf{w}_2(\mathbf{x},t)\big] \, q_{2}(\mathbf{x},t) \; \mathrm{d} \Omega
\end{align}
Substituting the mass balance equations \eqref{Eqn:Backwards_wq_BoM_1} and \eqref{Eqn:Backwards_wq_BoM_2} into the above equation, we arrive at the following:
\begin{align}
    \dot{\mathcal{E}}(t) 
    &= -\int_{\Omega} 
    \mu \, \mathbf{K}_1^{-1}(\mathbf{x}) \, \mathbf{w}_1(\mathbf{x},t) 
    \bullet \mathbf{w}_{1}(\mathbf{x},t) \; \mathrm{d} \Omega 
    -\int_{\Omega} 
    \mu \, \mathbf{K}_2^{-1}(\mathbf{x}) \, \mathbf{w}_2(\mathbf{x},t) 
    \bullet \mathbf{w}_{2}(\mathbf{x},t) \; \mathrm{d} \Omega 
    \nonumber \\
    &\qquad \qquad -\int_{\Omega} 
    \frac{\beta}{\mu} \, \big(q_1(\mathbf{x},t) - q_2(\mathbf{x},t)\big)^2 \; 
    \mathrm{d} \Omega
\end{align}

Now taking the time derivative on both sides, we establish: 
\begin{align}
    \ddot{\mathcal{E}}(t) 
    &= -2 \int_{\Omega} 
    \mu \, \mathbf{K}_1^{-1}(\mathbf{x}) \, \mathbf{w}_1(\mathbf{x},t) 
    \bullet \frac{\partial \mathbf{w}_{1}(\mathbf{x},t)}{\partial t} \; \mathrm{d} \Omega 
    -2 \int_{\Omega} 
    \mu \, \mathbf{K}_2^{-1}(\mathbf{x}) \, \mathbf{w}_2(\mathbf{x},t) 
    \bullet \frac{\partial \mathbf{w}_{2}(\mathbf{x},t)}{\partial t} \; \mathrm{d} \Omega 
    \nonumber \\
    &\qquad -2\int_{\Omega} 
    \frac{\beta}{\mu} \, \big(q_1(\mathbf{x},t) - q_2(\mathbf{x},t)\big) \left(\frac{\partial q_1(\mathbf{x},t)}{\partial t} - \frac{\partial q_2(\mathbf{x},t)}{\partial t} \right) \; 
    \mathrm{d} \Omega
\end{align}
Using the mass balance equations \eqref{Eqn:Backwards_wq_BoM_1} and \eqref{Eqn:Backwards_wq_BoM_2}, we write: 
\begin{align}
    \ddot{\mathcal{E}}(t) 
    &= -2 \int_{\Omega} 
    \mu \, \mathbf{K}_1^{-1}(\mathbf{x}) \, \mathbf{w}_1(\mathbf{x},t) 
    \bullet \frac{\partial \mathbf{w}_{1}(\mathbf{x},t)}{\partial t} \; \mathrm{d} \Omega 
    -2 \int_{\Omega} 
    \mu \, \mathbf{K}_2^{-1}(\mathbf{x}) \, \mathbf{w}_2(\mathbf{x},t) 
    \bullet \frac{\partial \mathbf{w}_{2}(\mathbf{x},t)}{\partial t} \; \mathrm{d} \Omega 
    \nonumber \\
    &\qquad +2\int_{\Omega} 
    q_1(\mathbf{x},t) \, \mathrm{div}\Big[\frac{\partial \mathbf{w}_1(\mathbf{x},t)}{\partial t} \Big] \; 
    \mathrm{d} \Omega
    + 2\int_{\Omega} 
    q_2(\mathbf{x},t) \, \mathrm{div}\Big[\frac{\partial \mathbf{w}_2(\mathbf{x},t)}{\partial t}\Big] \; 
    \mathrm{d} \Omega
\end{align}
Utilizing Green's identity and noting the boundary partitions given by Eqs.~\eqref{Eqn:Backwards_macro_gamma_partition} and \eqref{Eqn:Backwards_micro_gamma_partition}, the above equation becomes: 
\begin{align}
    \ddot{\mathcal{E}}(t) 
    &= -2 \int_{\Omega} 
    \Big\{
    \mu \, \mathbf{K}_1^{-1}(\mathbf{x}) \, \mathbf{w}_1(\mathbf{x},t) + 
    \mathrm{grad}\big[q_1(\mathbf{x},t)\big]
    \Big\} 
    \bullet \frac{\partial \mathbf{w}_{1}(\mathbf{x},t)}{\partial t} \; \mathrm{d} \Omega \nonumber \\
    &\qquad -2 \int_{\Omega} 
    \Big\{
    \mu \, \mathbf{K}_2^{-1}(\mathbf{x}) \, \mathbf{w}_2(\mathbf{x},t) + 
    \mathrm{grad}\big[q_2(\mathbf{x},t)\big]
    \Big\} 
    \bullet \frac{\partial \mathbf{w}_{2}(\mathbf{x},t)}{\partial t} \; \mathrm{d} \Omega
    \nonumber \\ 
    &\qquad +2 \int_{\Gamma_1^{u}}  
    q_1(\mathbf{x},t) \, \left(
    \frac{\partial \mathbf{w}_{1}(\mathbf{x},t)}{\partial t} 
    \bullet \widehat{\mathbf{n}}(\mathbf{x}) 
    \right) \; \mathrm{d} \Gamma
    +2 \int_{\Gamma_1^{p}}  
    q_1(\mathbf{x},t) \, \left(
    \frac{\partial \mathbf{w}_{1}(\mathbf{x},t)}{\partial t} 
    \bullet \widehat{\mathbf{n}}(\mathbf{x}) 
    \right) \; \mathrm{d} \Gamma
    \nonumber \\
    &\qquad +2 \int_{\Gamma_2^{u}}  
    q_2(\mathbf{x},t) \, \left(
    \frac{\partial \mathbf{w}_{2}(\mathbf{x},t)}{\partial t} 
    \bullet \widehat{\mathbf{n}}(\mathbf{x}) 
    \right) \; \mathrm{d} \Gamma
    +2 \int_{\Gamma_2^{p}}  
    q_2(\mathbf{x},t) \, \left(
    \frac{\partial \mathbf{w}_{2}(\mathbf{x},t)}{\partial t} 
    \bullet \widehat{\mathbf{n}}(\mathbf{x}) 
    \right) \; \mathrm{d} \Gamma
\end{align}
Once again, appealing to the boundary conditions \eqref{Eqn:Backwards_wq_vBC_1}--\eqref{Eqn:Backwards_wq_pBC_2}, we conclude that the last four boundary integrals are zero, resulting in the following:
\begin{align}
    \ddot{\mathcal{E}}(t) 
    &= -2 \int_{\Omega} 
    \Big\{
    \mu \, \mathbf{K}_1^{-1}(\mathbf{x}) \, \mathbf{w}_1(\mathbf{x},t) + 
    \mathrm{grad}\big[q_1(\mathbf{x},t)\big]
    \Big\} 
    \bullet \frac{\partial \mathbf{w}_{1}(\mathbf{x},t)}{\partial t} \; \mathrm{d} \Omega \nonumber \\
    &\qquad -2 \int_{\Omega} 
    \Big\{
    \mu \, \mathbf{K}_2^{-1}(\mathbf{x}) \, \mathbf{w}_2(\mathbf{x},t) + 
    \mathrm{grad}\big[q_2(\mathbf{x},t)\big]
    \Big\} 
    \bullet \frac{\partial \mathbf{w}_{2}(\mathbf{x},t)}{\partial t} \; \mathrm{d} \Omega
\end{align}
Finally, invoking Eqs.~\eqref{Eqn:Backwards_wq_BoLM_1} and \eqref{Eqn:Backwards_wq_BoLM_2} gives the second identity: 
\begin{align}
    \ddot{\mathcal{E}}(t) 
    = 2 \int_{\Omega} \frac{\gamma}{\phi_1(\mathbf{x})} \frac{\partial \mathbf{w}_{1}(\mathbf{x},t)}{\partial t} \bullet \frac{\partial \mathbf{w}_{1}(\mathbf{x},t)}{\partial t} \; \mathrm{d} \Omega 
    + 2 \int_{\Omega} \frac{\gamma}{\phi_2(\mathbf{x})} \frac{\partial \mathbf{w}_{2}(\mathbf{x},t)}{\partial t} \bullet \frac{\partial \mathbf{w}_{2}(\mathbf{x},t)}{\partial t} \; \mathrm{d} \Omega 
\end{align}
\end{proof}

\begin{proposition}
\label{Prop:Backward_Inequality_for_E}
The function $\mathcal{E}(t)$, defined in Eq.~\eqref{Eqn:Backwards_E_of_t}, satisfies the following inequality: 
\begin{align}
\label{Eqn:Backwards_Important_inequality}
\left(\dot{\mathcal{E}}(t)\right)^2 
\leq \mathcal{E}(t) \, \ddot{\mathcal{E}}(t) 
\end{align}
\end{proposition}
\begin{proof}
We start with the expression for $\dot{\mathcal{E}}(t)$ (i.e., Eq.~\eqref{Eqn:Backwards_dE_by_dt}): 
\begin{align}
    \dot{\mathcal{E}}(t) 
    = \int_{\Omega} \frac{\gamma}{\phi_1(\mathbf{x})} \frac{\partial \mathbf{w}_{1}(\mathbf{x},t)}{\partial t} \bullet \mathbf{w}_{1}(\mathbf{x},t) \; \mathrm{d} \Omega
    + \int_{\Omega} \frac{\gamma}{\phi_2(\mathbf{x})} \frac{\partial \mathbf{w}_{2}(\mathbf{x},t)}{\partial t} \bullet \mathbf{w}_{2}(\mathbf{x},t) \; \mathrm{d} \Omega 
\end{align}
By squaring both sides of the above equation, we write the following: 
\begin{align}
    \left(\dot{\mathcal{E}}(t)\right)^2 
    &= \left(\int_{\Omega} \frac{\gamma}{\phi_1(\mathbf{x})} \frac{\partial \mathbf{w}_{1}(\mathbf{x},t)}{\partial t} \bullet \mathbf{w}_{1}(\mathbf{x},t) \; \mathrm{d} \Omega\right)^2 
    + \left(\int_{\Omega} \frac{\gamma}{\phi_2(\mathbf{x})} \frac{\partial \mathbf{w}_{2}(\mathbf{x},t)}{\partial t} \bullet \mathbf{w}_{2}(\mathbf{x},t) \; \mathrm{d} \Omega \right)^2 
    \nonumber \\
    &\qquad + 2 
    \left(\int_{\Omega} \frac{\gamma}{\phi_1(\mathbf{x})} \frac{\partial \mathbf{w}_{1}(\mathbf{x},t)}{\partial t} \bullet \mathbf{w}_{1}(\mathbf{x},t) \; \mathrm{d} \Omega \right)
    \left(\int_{\Omega} \frac{\gamma}{\phi_2(\mathbf{x})} \frac{\partial \mathbf{w}_{2}(\mathbf{x},t)}{\partial t} \bullet \mathbf{w}_{2}(\mathbf{x},t)
    \; \mathrm{d} \Omega\right)  
\end{align}
Applying Cauchy-Schwarz inequality \citep{steele2004cauchy} individually to the terms in the brackets on the right side of the above equation, we get: 
\begin{align}
    \left(\dot{\mathcal{E}}(t)\right)^2  
    &\leq 
    \left(\int_{\Omega} \frac{\gamma}{\phi_1(\mathbf{x})} \mathbf{w}_{1}(\mathbf{x},t) \bullet \mathbf{w}_{1}(\mathbf{x},t) \; \mathrm{d} \Omega \right)
    \left(\int_{\Omega} \frac{\gamma}{\phi_1(\mathbf{x})} \frac{\partial \mathbf{w}_{1}(\mathbf{x},t)}{\partial t} \bullet \frac{\partial \mathbf{w}_{1}(\mathbf{x},t)}{\partial t} \; \mathrm{d} \Omega\right) 
     \nonumber \\
    &\qquad + 
    \left(\int_{\Omega} \frac{\gamma}{\phi_2(\mathbf{x})} \mathbf{w}_{2}(\mathbf{x},t) \bullet \mathbf{w}_{2}(\mathbf{x},t) \; \mathrm{d} \Omega \right)
    \left(\int_{\Omega} \frac{\gamma}{\phi_2(\mathbf{x})} \frac{\partial \mathbf{w}_{2}(\mathbf{x},t)}{\partial t} \bullet \frac{\partial \mathbf{w}_{2}(\mathbf{x},t)}{\partial t} \; \mathrm{d} \Omega\right) 
     \nonumber \\
    &\qquad + 2 \; \sqrt{
    \left(\int_{\Omega} \frac{\gamma}{\phi_1(\mathbf{x})} \mathbf{w}_{1}(\mathbf{x},t) \bullet \mathbf{w}_{1}(\mathbf{x},t) \; \mathrm{d} \Omega \right)
    \left(\int_{\Omega} \frac{\gamma}{\phi_1(\mathbf{x})} \frac{\partial \mathbf{w}_{1}(\mathbf{x},t)}{\partial t} \bullet \frac{\partial \mathbf{w}_{1}(\mathbf{x},t)}{\partial t} \; \mathrm{d} \Omega\right)
    } \nonumber \\
    &\qquad \qquad \sqrt{
    \left(\int_{\Omega} \frac{\gamma}{\phi_2(\mathbf{x})} \mathbf{w}_{2}(\mathbf{x},t) \bullet \mathbf{w}_{2}(\mathbf{x},t) \; \mathrm{d} \Omega \right)
    \left(\int_{\Omega} \frac{\gamma}{\phi_2(\mathbf{x})} \frac{\partial \mathbf{w}_{2}(\mathbf{x},t)}{\partial t} \bullet \frac{\partial \mathbf{w}_{2}(\mathbf{x},t)}{\partial t} \; \mathrm{d} \Omega\right)
    }
\end{align}
Invoking the \emph{inequality of arithmetic and geometric means} (i.e., A.M.-G.M. inequality) \citep{hardy1952inequalities} on the last term, we have the following estimate: 
\begin{align}
    \left(\dot{\mathcal{E}}(t)\right)^2
    &\leq 
    \left(\int_{\Omega} \frac{\gamma}{\phi_1(\mathbf{x})} \mathbf{w}_{1}(\mathbf{x},t) \bullet \mathbf{w}_{1}(\mathbf{x},t) \; \mathrm{d} \Omega \right)
    \left(\int_{\Omega} \frac{\gamma}{\phi_1(\mathbf{x})} \frac{\partial \mathbf{w}_{1}(\mathbf{x},t)}{\partial t} \bullet \frac{\partial \mathbf{w}_{1}(\mathbf{x},t)}{\partial t} \; \mathrm{d} \Omega\right) 
     \nonumber \\
    &\qquad + 
    \left(\int_{\Omega} \frac{\gamma}{\phi_2(\mathbf{x})} \mathbf{w}_{2}(\mathbf{x},t) \bullet \mathbf{w}_{2}(\mathbf{x},t) \; \mathrm{d} \Omega \right)
    \left(\int_{\Omega} \frac{\gamma}{\phi_2(\mathbf{x})} \frac{\partial \mathbf{w}_{2}(\mathbf{x},t)}{\partial t} \bullet \frac{\partial \mathbf{w}_{2}(\mathbf{x},t)}{\partial t} \; \mathrm{d} \Omega\right) 
     \nonumber \\
    &\qquad + 
    \left(\int_{\Omega} \frac{\gamma}{\phi_1(\mathbf{x})} \mathbf{w}_{1}(\mathbf{x},t) \bullet \mathbf{w}_{1}(\mathbf{x},t) \; \mathrm{d} \Omega \right)
    \left(\int_{\Omega} \frac{\gamma}{\phi_2(\mathbf{x})} \frac{\partial \mathbf{w}_{2}(\mathbf{x},t)}{\partial t} \bullet \frac{\partial \mathbf{w}_{2}(\mathbf{x},t)}{\partial t} \; \mathrm{d} \Omega\right) 
     \nonumber \\
    &\qquad + 
    \left(\int_{\Omega} \frac{\gamma}{\phi_2(\mathbf{x})} \mathbf{w}_{2}(\mathbf{x},t) \bullet \mathbf{w}_{2}(\mathbf{x},t) \; \mathrm{d} \Omega \right)
    \left(\int_{\Omega} \frac{\gamma}{\phi_1(\mathbf{x})} \frac{\partial \mathbf{w}_{1}(\mathbf{x},t)}{\partial t} \bullet \frac{\partial \mathbf{w}_{1}(\mathbf{x},t)}{\partial t} \; \mathrm{d} \Omega\right) 
\end{align}
We arrange the terms on the right side of the above equation to obtain the following: 
\begin{align}
    \left(\dot{\mathcal{E}}(t)\right)^2  
    &\leq \left(
    \int_{\Omega} \frac{\gamma}{\phi_1(\mathbf{x})} \mathbf{w}_{1}(\mathbf{x},t) \bullet \mathbf{w}_{1}(\mathbf{x},t) \, \mathrm{d} \Omega  
    + \int_{\Omega} \frac{\gamma}{\phi_2(\mathbf{x})} \mathbf{w}_{2}(\mathbf{x},t) \bullet \mathbf{w}_{2}(\mathbf{x},t) \; \mathrm{d} \Omega 
    \right) 
    \nonumber \\
    &\qquad  
    \left(\int_{\Omega} \frac{\gamma}{\phi_1(\mathbf{x})} \frac{\partial \mathbf{w}_{1}(\mathbf{x},t)}{\partial t} \bullet \frac{\partial \mathbf{w}_{1}(\mathbf{x},t)}{\partial t} \; \mathrm{d} \Omega 
    + \int_{\Omega} \frac{\gamma}{\phi_2(\mathbf{x})} \frac{\partial \mathbf{w}_{2}(\mathbf{x},t)}{\partial t} \bullet \frac{\partial \mathbf{w}_{2}(\mathbf{x},t)}{\partial t} \; \mathrm{d} \Omega\right)  
\end{align}
Using Proposition \ref{Prop:Backward_Time_deriatives_of_E} and noting the factors ``$1/2$'' and ``$2$'' in Eqs.~\eqref{Eqn:Backwards_dE_by_dt} and \eqref{Eqn:Backwards_d2E_by_dt2}, we arrive at the desired inequality:
\begin{align}
    \left(\dot{\mathcal{E}}(t)\right)^2  
    & \leq 
    \Big(2 \, \mathcal{E}(t) \Big) \, \Big(\frac{1}{2}\ddot{\mathcal{E}}(t) \Big)
    = \mathcal{E}(t) \, \ddot{\mathcal{E}}(t) 
\end{align}
\end{proof}

\begin{remark}
 The classical heat equation also enjoys an inequality of the form given by Eq.~\eqref{Eqn:Backwards_Important_inequality}. Moreover, proofs establishing the backward-in-time uniqueness for the classical heat equation often utilize such an inequality; for example, see \citep[pages 63--65]{evans2002partial}. Of course, the expression for $\mathcal{E}(t)$ is different for the heat equation. 
\end{remark}

\begin{proposition}
\label{Prop:Backwards_Exponential_inequality}
In an interval $[t_1, t_2)$, if a function $\mathcal{F}(t)$ satisfies 
\begin{align}
\label{Eqn:Backwards_Intermediate_result_hypothesis}
\mathcal{F}(t)  > 0 \quad \mathrm{and} \quad  
\left(\dot{\mathcal{F}}(t)\right)^2 \leq \mathcal{F}(t) \, \ddot{\mathcal{F}}(t)
\end{align}
then 
\begin{align}
\mathcal{F}(t) \geq  \mathcal{F}(t_1) \, 
\exp\left[ \frac{\dot{\mathcal{F}}(t_1) \, (t - t_1)}{\mathcal{F}(t_1)}\right] \quad \forall t \in [t_1,t_2)
\end{align}
\end{proposition}
\begin{proof}
Consider the following function: 
\begin{align}
\label{Eqn:Backwards_Function_G_of_t}
\mathcal{G}(t) 
:= \frac{\dot{\mathcal{F}}(t)}{\mathcal{F}(t)}
\end{align}
Since $\mathcal{F}(t) > 0$, $\mathcal{G}(t)$ is well defined. Taking the time derivative on both sides of the above equations, we obtain the following:
\begin{align}
\dot{\mathcal{G}}(t) := \frac{\mathcal{F}(t) \, \ddot{\mathcal{F}}(t) - \big(\dot{\mathcal{F}}(t)\big)^2}{\big(\mathcal{F}(t)\big)^2}
\end{align}
By virtue of Eq.~\eqref{Eqn:Backwards_Intermediate_result_hypothesis}, we conclude that 
\begin{align}
\dot{\mathcal{G}}(t) \geq 0 \quad \forall t \in [t_1,t_2)
\end{align}

Using the properties of integration (e.g., see \citep{bartle2000introduction}), we infer:
\begin{align}
\mathcal{G}(t) - \mathcal{G}(t_1) = \int_{t_1}^{t} \dot{\mathcal{G}}(t) \, \mathrm{d} t \geq 0 
\quad \forall t \in [t_1,t_2)
\end{align}
which, on the account of Eq.~\eqref{Eqn:Backwards_Function_G_of_t}, further implies that 
\begin{align}
\frac{\dot{\mathcal{F}}(t)}{\mathcal{F}(t)} = \mathcal{G}(t) \geq \mathcal{G}(t_1) = \frac{\dot{\mathcal{F}}(t_1)}{\mathcal{F}(t_1)}
\quad \forall t \in [t_1,t_2)
\end{align}
Noting $\mathcal{F}(t) > 0$, we establish the following inequality: 
\begin{align}
\dot{\mathcal{F}}(t) \geq  \frac{\dot{\mathcal{F}}(t_1)}{\mathcal{F}(t_1)} \, \mathcal{F}(t)
\quad \forall t \in [t_1,t_2)
\end{align}
Invoking the Gronwall-Bellman inequality \citep{pachpatte1998inequalities}, we arrive at the desired estimate: 
\begin{align}
\mathcal{F}(t) \geq  \mathcal{F}(t_1) \, 
\exp\left[ \frac{\dot{\mathcal{F}}(t_1) \, (t - t_1)}{\mathcal{F}(t_1)}\right] \quad \forall t \in [t_1,t_2)
\end{align}
\end{proof}

We now return to the proof of the backward-in-time uniqueness theorem. 

\begin{figure}
    \includegraphics[scale=1.75]{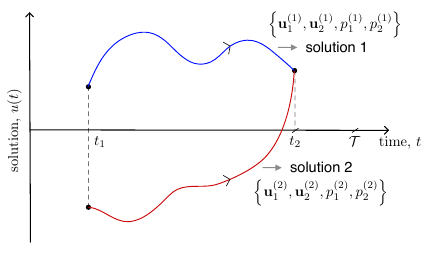}
    \caption{Illustrating the non-uniqueness of backward-in-time solutions. \label{Fig:Fig2_nonuniqueness}}
\end{figure}

\begin{proof}[\textbf{Proof of Theorem \ref{Thm:Backwards_Backward_uniqueness_theorem}}]
By virtue of Eq.~\eqref{Eqn:Backwards_theorem_condition_1}, we note: 
\begin{align}
    \mathbf{w}_{1}(\mathbf{x},\mathcal{T}) = \mathbf{0} 
    \quad \mathrm{and} \quad 
    \mathbf{w}_{2}(\mathbf{x},\mathcal{T}) = \mathbf{0}
    \quad \forall \mathbf{x} \in \Omega 
\end{align}

If $\mathbf{w}_1(\mathbf{x},t) = \mathbf{0}$ and $\mathbf{w}_2(\mathbf{x},t) = \mathbf{0}$ in the entire time interval $0 \leq t < \mathcal{T}$, we are done with the proof. On the contrary, if the two solutions are not identical, then there must exist a sub-interval $[t_1,t_2) \subset [0,\mathcal{T})$ such that 
\begin{align}
    \mathrm{either} \quad 
    \mathbf{w}_1(\mathbf{x},t) \neq \mathbf{0} 
    \quad \mathrm{or} \quad
    \mathbf{w}_2(\mathbf{x},t) \neq \mathbf{0}
\end{align}
with
\begin{align}
\mathbf{w}_1(\mathbf{x},t_2) = \mathbf{0} 
\quad \mathrm{and} \quad
\mathbf{w}_2(\mathbf{x},t_2) = \mathbf{0}
\end{align}
See \textbf{Fig.~\ref{Fig:Fig2_nonuniqueness}} for a pictorial description of the conjectured  non-uniqueness.  

Noting the definition for $\mathcal{E}(t)$ given by Eq.~\eqref{Eqn:Backwards_E_of_t}, the above two equations imply that 
\begin{align}
\label{Eqn:Backwards_E_of_t_in_t1_to_t2}
\mathcal{E}(t) > 0 \quad \forall t \in [t_1,t_2) 
\quad \mathrm{and} \quad 
\mathcal{E}(t_2) = 0
\end{align}
From Proposition \ref{Prop:Backward_Inequality_for_E} we have: 
\[
\left(\dot{\mathcal{E}}(t)\right)^2 \leq \mathcal{E}(t) \, \ddot{\mathcal{E}}(t)
\]
Consequently, Proposition \ref{Prop:Backwards_Exponential_inequality} offers the following inequality: 
\begin{align}
\mathcal{E}(t) \geq  \mathcal{E}(t_1) \, 
\exp\left[ \frac{\dot{\mathcal{E}}(t_1) \, (t - t_1)}{\mathcal{E}(t_1)}\right]  \quad \forall t \in [t_1,t_2)
\end{align}
Passing the limit $t \rightarrow t_2$ on both sides of the above inequality, and noting that $\mathcal{E}(t)>0$ and exponential function in positive, we infer that   
\begin{align}
\mathcal{E}(t_2) \geq  \mathcal{E}(t_1) \, 
\exp\left[ \frac{\dot{\mathcal{E}}(t_1) \, (t_2 - t_1)}{\mathcal{E}(t_1)}\right] > 0
\end{align}
which is a contradiction with the condition $\mathcal{E}(t_2) = 0$ (i.e., Eq.~\eqref{Eqn:Backwards_E_of_t_in_t1_to_t2}). Hence, we conclude that  
\begin{align}
\mathbf{w}_1(\mathbf{x},t) = \mathbf{0} 
\quad \mathrm{and} \quad 
\mathbf{w}_2(\mathbf{x},t) = \mathbf{0} 
\quad \forall \mathbf{x} \in \overline{\Omega}, \forall t \in [0,\mathcal{T}] 
\end{align}
meaning that 
\begin{align}
\mathbf{u}_1^{(1)}(\mathbf{x},t) = \mathbf{u}_1^{(2)}(\mathbf{x},t)
\quad \mathrm{and} \quad 
\mathbf{u}_2^{(1)}(\mathbf{x},t) = \mathbf{u}_2^{(2)}(\mathbf{x},t)
\quad \forall \mathbf{x} \in \overline{\Omega}, \forall t \in [0,\mathcal{T}] 
\end{align}

We now return to pressures. Since $\mathbf{w}_1(\mathbf{x},t) = \mathbf{0}$, Eq.~\eqref{Eqn:Backwards_wq_BoLM_1} necessitates that
\begin{align}
    \mathrm{grad}\big[q_1(\mathbf{x},t)\big] = \mathbf{0}
\end{align}
which further implies that 
\begin{align}
    q_1(\mathbf{x},t) = C_1 
\end{align}
Similarly, $\mathbf{w}_2(\mathbf{x},t) = \mathbf{0}$ and Eq.~\eqref{Eqn:Backwards_wq_BoLM_2} entail 
\begin{align}
    q_2(\mathbf{x},t) = C_2 
\end{align}
Equation \eqref{Eqn:Backwards_wq_BoM_1} requires that 
\begin{align}
    C_1 = C_2 = C
\end{align}
We have thus established that 
\begin{align} 
p_1^{(1)}(\mathbf{x},t) = p_1^{(2)}(\mathbf{x},t) + C 
\quad \mathrm{and} \quad 
p_2^{(1)}(\mathbf{x},t) = p_2^{(2)}(\mathbf{x},t) + C 
\end{align} 
If $\Gamma_1^{p} \neq 0$, then $C = 0$ on the account of Eq.~\eqref{Eqn:Backwards_wq_pBC_1}. Alternatively, if $\Gamma_2^{p} \neq 0$, then $C = 0$ because of Eq.~\eqref{Eqn:Backwards_wq_pBC_2}. This completes the proof.
\end{proof}

\subsection{A remark on $\mathcal{E}(t)$}
\label{Subsec:Backwards_Remark_on_E_of_t}
Since $\mathcal{E}(t)$ played a crucial role in establishing the uniqueness of backward-in-time solutions, it is essential to profess the quantity is not a mysterious mathematical object but has strong mechanics underpinning. $\mathcal{E}(t)$ is based on the total kinetic energy of the fluid. If we take  $\mathbf{w}_1(\mathbf{x},t) = \mathbf{u}_1(\mathbf{x},t)$ and $\mathbf{w}_2(\mathbf{x},t) = \mathbf{u}_2(\mathbf{x},t)$ in Eq.~\eqref{Eqn:Backwards_E_of_t}, the resulting quantity is the total kinetic energy. 

We now show that the following quantity:
\begin{align}
    \label{Eqn:Backwards_K_of_t}
    \mathcal{K}(t)   
    = \int_{\Omega} \frac{1}{2} \frac{\gamma}{\phi_1(\mathbf{x})} \mathbf{u}_{1}(\mathbf{x},t) \bullet \mathbf{u}_{1}(\mathbf{x},t) \; \mathrm{d} \Omega 
    + \int_{\Omega} \frac{1}{2} \frac{\gamma}{\phi_2(\mathbf{x})} \mathbf{u}_{2}(\mathbf{x},t) \bullet \mathbf{u}_{2}(\mathbf{x},t) \; \mathrm{d} \Omega 
\end{align}
is, in fact, the total kinetic energy accounting flow in both networks.
The porosity factors (i.e., $\phi_1(\mathbf{x})$ and $\phi_2(\mathbf{x})$) do not make the said connection apparent. To clarify, we need to relate the seepage (i.e., true) velocities with the discharge (i.e., Darcy) velocities. Seepage velocity is the actual velocity of the fluid in the pore. If $\mathbf{v}_1(\mathbf{x},t)$ and $\mathbf{v}_2(\mathbf{x},t)$ denote the seepage velocities in the macro- and micro-networks, then these quantities are related to their counterpart discharge velocities as follows:
\begin{align}
\mathbf{u}_{1}(\mathbf{x},t) = \phi_1(\mathbf{x}) \, \mathbf{v}_1(\mathbf{x},t)  
\quad \mathrm{and} \quad 
\mathbf{u}_{2}(\mathbf{x},t) = \phi_2(\mathbf{x}) \, \mathbf{v}_2(\mathbf{x},t)
\end{align}
Also, the differential volume occupied by the macro-pore within a differential volume $\mathrm{d}\Omega$ is given by 
\begin{align}
\mathrm{d} \Omega_{\mathrm{macro}} = \phi_1(\mathbf{x}) \, 
\mathrm{d} \Omega
\end{align}
Likewise, the differential volume occupied by the micro-pore is 
\begin{align}
\mathrm{d} \Omega_{\mathrm{micro}} = \phi_2(\mathbf{x}) \, 
\mathrm{d} \Omega
\end{align}

By substituting the above equations, the first integral in Eq.~\eqref{Eqn:Backwards_K_of_t} can be written as follows: 
\begin{align}
\int_{\Omega} \frac{1}{2} \frac{\gamma}{\phi_1(\mathbf{x})} \mathbf{u}_{1}(\mathbf{x},t) \bullet \mathbf{u}_{1}(\mathbf{x},t) \; \mathrm{d} \Omega 
&= \int_{\Omega} \frac{1}{2} \phi_1(\mathbf{x}) \, \gamma \, \mathbf{v}_{1}(\mathbf{x},t) \bullet \mathbf{v}_{1}(\mathbf{x},t) \; \mathrm{d} \Omega
\nonumber \\
&= \int_{\Omega_{\mathrm{macro}}} \frac{1}{2} \, \gamma \, \mathbf{v}_{1}(\mathbf{x},t) \bullet \mathbf{v}_{1}(\mathbf{x},t) \; \mathrm{d} \Omega_{\mathrm{macro}} 
\end{align}
Clearly, the above expression is the kinetic energy of the fluid flowing through the macro-network. In like manner, the second integral amounts to   
\begin{align}
\int_{\Omega} \frac{1}{2} \frac{\gamma}{\phi_1(\mathbf{x})} \mathbf{u}_{2}(\mathbf{x},t) \bullet \mathbf{u}_{2}(\mathbf{x},t) \; \mathrm{d} \Omega 
&= \int_{\Omega_{\mathrm{micro}}} \frac{1}{2} \, \gamma \, \mathbf{v}_{2}(\mathbf{x},t) \bullet \mathbf{v}_{2}(\mathbf{x},t) \; \mathrm{d} \Omega_{\mathrm{micro}} 
\end{align}
which is the kinetic energy of the fluid flowing through the micro-network. Therefore, $\mathcal{K}(t)$ is the total kinetic energy of the fluid considering the flow in both pore-networks. 

The prior conversation is yet another demonstration of how \emph{mechanics} can aid in constructing proofs by selecting appropriate norms and estimating powerful bounds.

\section{A RECIPROCAL RELATION}
\label{Sec:S4_Backwards_Reciprocity}
For convenience, we group various quantities. We use the following notation: 
\begin{align}
    \mathcal{S}_{\mathrm{macro}} 
    := \Big\{\mathbf{u}_{1}(\mathbf{x},t),p_{1}(\mathbf{x},t)\Big\} 
    \quad \mathrm{and} \quad 
    \mathcal{S}_{\mathrm{micro}} 
    := \Big\{\mathbf{u}_{2}(\mathbf{x},t),p_{2}(\mathbf{x},t)\Big\} 
\end{align}
to denote collectively the \emph{solution} fields of the macro- and micro-networks, respectively. In like manner, we group the \emph{prescribed} quantities for the macro- and micro-networks as follows: 
\begin{align}
\mathcal{P}_{\mathrm{macro}} &:= 
\Big\{\mathbf{b}_{1}(\mathbf{x},t),\mathbf{u}_{01}(\mathbf{x}),u_{n1}(\mathbf{x},t),p_{\mathrm{p}1}(\mathbf{x},t)\Big\} 
\quad \mathrm{and} \quad \nonumber \\
&\hspace{1in} \mathcal{P}_{\mathrm{micro}} := 
\Big\{\mathbf{b}_{2}(\mathbf{x},t),\mathbf{u}_{02}(\mathbf{x}),u_{n2}(\mathbf{x},t),p_{\mathrm{p}2}(\mathbf{x},t)\Big\} 
\end{align}
We also use the following notation to group various integrals of macro- and micro-networks: 
\begin{align}
    \label{Eqn:Backwards_macro_integrals}
    \Big\langle \mathcal{S}_{\mathrm{macro}};\mathcal{P}_{\mathrm{macro}}\Big\rangle_{\mathrm{macro}} 
    &:= \int_{\Omega} \mathbf{u}_{1} \star  
    \gamma \, \mathbf{b}_{1} \; \mathrm{d} \Omega 
    + \int_{\Omega} \frac{\gamma}{\phi_{1}(\mathbf{x})} \mathbf{u}_{1}(\mathbf{x},t) 
    \bullet  \mathbf{u}_{01}(\mathbf{x}) \; \mathrm{d} \Omega \nonumber \\
    &\qquad -\int_{\Gamma^p_1}
    \big(\mathbf{u}_{1} \bullet \widehat{\mathbf{n}}(\mathbf{x}) \big) 
    \star p_{\mathrm{p}1}  \; \mathrm{d} \Gamma 
    +\int_{\Gamma^u_1}
    p_{1} \star u_{n1} \; \mathrm{d} \Gamma \\
    \label{Eqn:Backwards_micro_integrals}
    \Big\langle \mathcal{S}_{\mathrm{micro}};\mathcal{P}_{\mathrm{micro}}\Big\rangle_{\mathrm{micro}} 
    &:= \int_{\Omega} \mathbf{u}_{2} \star  
    \gamma \, \mathbf{b}_{2} \; \mathrm{d} \Omega 
    + \int_{\Omega} \frac{\gamma}{\phi_{2}(\mathbf{x})} \mathbf{u}_{2}(\mathbf{x},t) 
    \bullet \mathbf{u}_{02}(\mathbf{x}) \; \mathrm{d} \Omega \nonumber \\
    &\qquad -\int_{\Gamma^p_2}
    \big(\mathbf{u}_{2} \bullet \widehat{\mathbf{n}}(\mathbf{x}) \big) 
    \star  p_{\mathrm{p}2}  \; \mathrm{d} \Gamma 
    +\int_{\Gamma^u_2}
    p_{2} \star u_{n2} \; \mathrm{d} \Gamma 
\end{align}

The reciprocal relation deals with two solutions corresponding to two different sets of prescribed quantities. To distinguish the two solutions and associated prescribed quantities, we use superscripts ``(1)'' and ``(2).'' For example, we use the following notation 
\begin{align}
    \mathcal{S}^{(1)}_{\mathrm{macro}} \equiv  
    \Big\{\mathbf{u}^{(1)}_1(\mathbf{x},t),p^{(1)}_1(\mathbf{x},t)\Big\}
\end{align}
to denote the \emph{macro} fields in the first solution. Likewise, 
\begin{align}
    \mathcal{S}^{(1)}_{\mathrm{micro}} \equiv  
    \Big\{\mathbf{u}^{(1)}_2(\mathbf{x},t),p^{(1)}_2(\mathbf{x},t)\Big\}
\end{align}
denotes the \emph{micro} fields in the first solution. Similarly, the prescribed quantities under the first solution are grouped as follows: 
\begin{align}
\mathcal{P}_{\mathrm{macro}}^{(1)} \equiv  
\Big\{\mathbf{b}_{1}^{(1)}(\mathbf{x},t),\mathbf{u}_{01}^{(1)}(\mathbf{x}),u_{n1}^{(1)}(\mathbf{x},t),p_{\mathrm{p}1}^{(1)}(\mathbf{x},t)\Big\} 
\end{align}
Analogous notations for $\mathcal{S}^{(2)}_{\mathrm{macro}}$, $\mathcal{S}^{(2)}_{\mathrm{micro}}$, $\mathcal{P}^{(1)}_{\mathrm{macro}}$, and $\mathcal{P}^{(1,2)}_{\mathrm{micro}}$ are manifest. With the above-introduced notation, the theorem below establishes the reciprocity for the DPP model in the transient regime. 

\begin{theorem}[Dynamic reciprocity]
\label{Thm:Backwards_Reciprocal_relation}
Let 
\begin{align*}
    \Big\{\mathcal{S}^{(1)}_{\mathrm{macro}},\mathcal{S}^{(1)}_{\mathrm{micro}}\Big\} 
    \quad \mathrm{and} \quad 
    \Big\{\mathcal{S}^{(2)}_{\mathrm{macro}},\mathcal{S}^{(2)}_{\mathrm{micro}}\Big\} 
\end{align*}
be the solutions corresponding to the two sets of prescribed quantities: 
\begin{align*}
    \Big\{\mathcal{P}^{(1)}_{\mathrm{macro}},\mathcal{P}^{(1)}_{\mathrm{micro}}\Big\} 
    \quad \mathrm{and} \quad 
    \Big\{\mathcal{P}^{(2)}_{\mathrm{macro}},\mathcal{P}^{(2)}_{\mathrm{micro}}\Big\} 
\end{align*}
Then the following relation holds: 
\begin{align}
    \label{Eqn:Backwards_Reciprocal_relation}
    \Big\langle \mathcal{S}^{(2)}_{\mathrm{macro}};
    \mathcal{P}^{(1)}_{\mathrm{macro}}
    \Big\rangle_{\mathrm{macro}} 
    + \Big\langle \mathcal{S}^{(2)}_{\mathrm{micro}};
    \mathcal{P}^{(1)}_{\mathrm{micro}}
    \Big\rangle_{\mathrm{micro}} 
    = \Big\langle \mathcal{S}^{(1)}_{\mathrm{macro}};
    \mathcal{P}^{(2)}_{\mathrm{macro}}
    \Big\rangle_{\mathrm{macro}} 
    + \Big\langle \mathcal{S}^{(1)}_{\mathrm{micro}};
    \mathcal{P}^{(2)}_{\mathrm{micro}}
    \Big\rangle_{\mathrm{micro}} 
\end{align}
\end{theorem}
\begin{proof}
The \emph{macro} fields of the \emph{first} solution (i.e., $\mathcal{S}_{\mathrm{macro}}^{(1)} \equiv \big\{\mathbf{u}^{(1)}_1(\mathbf{x},t),p^{(1)}_1(\mathbf{x},t)\big\}$) satisfy the following equation (cf. Eq.~\eqref{Eqn:Backwards_BoLM_1}): 
\begin{align}
\label{Eqn:Backwards_reciprocity_proof_step_1}
\frac{\gamma}{\phi_{1}(\mathbf{x})} 
\frac{\partial \mathbf{u}_{1}^{(1)}}{\partial t}
+ \mu \, \mathbf{K}_1^{-1}(\mathbf{x}) \, \mathbf{u}_1^{(1)}
+ \mathrm{grad}\big[p^{(1)}_{1}\big]
= \gamma \, \mathbf{b}_{1}^{(1)}
\end{align}
Applying convolution to the above equation with $\mathbf{u}_1^{(2)}(\mathbf{x},t)$, we write  
\begin{align}
\label{Eqn:Backwards_reciprocity_proof_step_2}
\mathbf{u}_{1}^{(2)} \star \gamma \, \mathbf{b}_{1}^{(1)}
= \mathbf{u}_{1}^{(2)} \star 
\frac{\gamma}{\phi_{1}(\mathbf{x})} 
\frac{\partial \mathbf{u}_{1}^{(1)}}{\partial t}
+ \mathbf{u}_{1}^{(2)} \star 
\mu \, \mathbf{K}_1^{-1}(\mathbf{x}) \, \mathbf{u}_1^{(1)}
+ \mathbf{u}_{1}^{(2)} \star 
\mathrm{grad}\big[p^{(1)}_{1}\big]
\end{align}
After adding the following term 
\[
\frac{\gamma}{\phi_1(\mathbf{x})} 
\mathbf{u}_1^{(2)}(\mathbf{x},t) \bullet 
\mathbf{u}_{01}^{(1)}(\mathbf{x})
\]
to both sides of Eq.~\eqref{Eqn:Backwards_reciprocity_proof_step_2}, we get the following: 
\begin{align}
\label{Eqn:Backwards_reciprocity_proof_step_3}
\mathbf{u}_{1}^{(2)} \star \gamma \, \mathbf{b}_{1}^{(1)}
+ \frac{\gamma}{\phi_1(\mathbf{x})} 
\mathbf{u}_1^{(2)}(\mathbf{x},t) \bullet 
\mathbf{u}_{01}^{(1)}(\mathbf{x})
&= \left\{\mathbf{u}_{1}^{(2)} \star 
\frac{\gamma}{\phi_{1}(\mathbf{x})} 
\frac{\partial \mathbf{u}_{1}^{(1)}}{\partial t}
+ \frac{\gamma}{\phi_1(\mathbf{x})} 
\mathbf{u}_1^{(2)}(\mathbf{x},t) \bullet 
\mathbf{u}_{01}^{(1)}(\mathbf{x})\right\} \nonumber \\
&\qquad + \mathbf{u}_{1}^{(2)} \star 
\mu \, \mathbf{K}_1^{-1}(\mathbf{x}) \, \mathbf{u}_1^{(1)}
+ \mathbf{u}_{1}^{(2)} \star 
\mathrm{grad}\big[p^{(1)}_{1}\big]
\end{align}

Carrying out similar steps (i.e., mimicking the calculations performed to get Eqs.~\eqref{Eqn:Backwards_reciprocity_proof_step_1}--\eqref{Eqn:Backwards_reciprocity_proof_step_3}) for the \emph{macro} fields of the \emph{second} solution (i.e., $\mathcal{S}_{\mathrm{macro}}^{(2)} \equiv \big\{\mathbf{u}^{(2)}_1(\mathbf{x},t),p^{(2)}_1(\mathbf{x},t)\big\}$), we arrive at the following: 
\begin{align}
\label{Eqn:Backwards_reciprocity_proof_step_4}
\mathbf{u}_{1}^{(1)} \star \gamma \, \mathbf{b}_{1}^{(2)}
+ \frac{\gamma}{\phi_1(\mathbf{x})} 
\mathbf{u}_1^{(1)}(\mathbf{x},t) \bullet 
\mathbf{u}_{01}^{(2)}(\mathbf{x})
&= \left\{\mathbf{u}_{1}^{(1)} \star 
\frac{\gamma}{\phi_{1}(\mathbf{x})} 
\frac{\partial \mathbf{u}_{1}^{(2)}}{\partial t}
+ \frac{\gamma}{\phi_1(\mathbf{x})} 
\mathbf{u}_1^{(1)}(\mathbf{x},t) \bullet 
\mathbf{u}_{01}^{(2)}(\mathbf{x})\right\} \nonumber \\
&\qquad + \mathbf{u}_{1}^{(1)} \star 
\mu \, \mathbf{K}_1^{-1}(\mathbf{x}) \, \mathbf{u}_1^{(2)}
+ \mathbf{u}_{1}^{(1)} \star 
\mathrm{grad}\big[p_{1}^{(2)}\big]
\end{align}

Subtracting the above two equations, we obtain: 
\begin{align}
\label{Eqn:Backwards_reciprocity_proof_step_5}
&\left\{
\mathbf{u}_{1}^{(2)} \star \gamma \, \mathbf{b}_{1}^{(1)}
+ \frac{\gamma}{\phi_1(\mathbf{x})} 
\mathbf{u}_1^{(2)}(\mathbf{x},t) \bullet 
\mathbf{u}_{01}^{(1)}(\mathbf{x})
\right\}
- \left\{
\mathbf{u}_{1}^{(1)} \star \gamma \, \mathbf{b}_{1}^{(2)}
+ \frac{\gamma}{\phi_1(\mathbf{x})} 
\mathbf{u}_1^{(1)}(\mathbf{x},t) \bullet 
\mathbf{u}_{01}^{(2)}(\mathbf{x})
\right\} \nonumber \\ 
&\qquad \qquad = \left\{\mathbf{u}_{1}^{(2)} \star 
\frac{\gamma}{\phi_{1}(\mathbf{x})} 
\frac{\partial \mathbf{u}_{1}^{(1)}}{\partial t} 
- \mathbf{u}_{1}^{(1)} \star 
\frac{\gamma}{\phi_{1}(\mathbf{x})} 
\frac{\partial \mathbf{u}_{1}^{(2)}}{\partial t} \right\} \nonumber \\ 
&\qquad \qquad \qquad + \left\{\mathbf{u}_{1}^{(2)} \star 
\mu \, \mathbf{K}_1^{-1}(\mathbf{x}) \, \mathbf{u}_1^{(1)} 
- \mathbf{u}_{1}^{(1)} \star 
\mu \, \mathbf{K}_1^{-1}(\mathbf{x}) \, \mathbf{u}_1^{(2)} \right\} \nonumber \\ 
&\qquad \qquad \qquad + \mathbf{u}_{1}^{(2)} \star 
\mathrm{grad}\big[p^{(1)}_{1}\big]
- \mathbf{u}_{1}^{(1)} \star 
\mathrm{grad}\big[p^{(2)}_{1}\big]
\end{align}
Equation \eqref{Eqn:Backwards_f_star_dgdt} implies that the first term in the curly brackets on the right side of Eq.~\eqref{Eqn:Backwards_reciprocity_proof_step_5} is zero: 
\begin{align}
\label{Eqn:Backwards_reciprocity_proof_step_6}
\mathbf{u}_{1}^{(2)} \star 
\frac{\gamma}{\phi_1(\mathbf{x})} 
\frac{\partial \mathbf{u}_1^{(1)}}{\partial t} 
- \mathbf{u}_{1}^{(1)} \star 
\frac{\gamma}{\phi_1(\mathbf{x})} 
\frac{\partial \mathbf{u}_1^{(1)}}{\partial t}
&= \mathbf{u}_{1}^{(2)}(\mathbf{x},0) \bullet
\frac{\gamma}{\phi_{1}(\mathbf{x})} 
\mathbf{u}_{1}^{(1)}(\mathbf{x},t)
- \mathbf{u}_{1}^{(1)}(\mathbf{x},t) 
\bullet 
\frac{\gamma}{\phi_{1}(\mathbf{x})} 
\mathbf{u}_{1}^{(2)}(\mathbf{x},0)  \nonumber \\ 
&= \mathbf{u}_{01}^{(2)}(\mathbf{x}) \bullet
\frac{\gamma}{\phi_{1}(\mathbf{x})} 
\mathbf{u}_{1}^{(1)}(\mathbf{x},t) 
- \mathbf{u}_{1}^{(1)}(\mathbf{x},t) 
\bullet 
\frac{\gamma}{\phi_{1}(\mathbf{x})} 
\mathbf{u}_{01}^{(2)}(\mathbf{x}) \nonumber \\
&= 0 
\end{align}
In view of the symmetry of the macro permeability tensor $\mathbf{K}_1(\mathbf{x})$, the commutative property of convolutions (i.e., Eq.~\eqref{Eqn:Backwards_commutative}), and the definition of the transpose of a tensor, the second term in curly brackets on the right side of Eq.~\eqref{Eqn:Backwards_reciprocity_proof_step_5} also vanishes. To wit, 
\begin{alignat}{2}
    \label{Eqn:Backwards_reciprocity_proof_step_7}
    \mathbf{u}_{1}^{(2)} \star 
    \mu \, \mathbf{K}_1^{-1}(\mathbf{x}) \, \mathbf{u}_1^{(1)} 
    &= \mu \, \mathbf{K}_1^{-\mathrm{T}}(\mathbf{x}) \, \mathbf{u}_{1}^{(2)} \star 
    \mathbf{u}_1^{(1)} 
    &&\quad \mbox{[definition of transpose]} \nonumber \\
    &= \mu \, \mathbf{K}_1^{-\mathrm{1}}(\mathbf{x}) \, \mathbf{u}_{1}^{(2)} \star 
    \mathbf{u}_1^{(1)} 
    &&\quad \mbox{[symmetry of $\mathbf{K}_1(\mathbf{x})$]} \nonumber \\
    &= \mathbf{u}_1^{(1)} \star \mu \, \mathbf{K}_1^{-\mathrm{1}}(\mathbf{x}) \, \mathbf{u}_{1}^{(2)} 
    &&\quad \mbox{[commutative property of convolutions]} 
\end{alignat}
On the account of Eqs.~\eqref{Eqn:Backwards_reciprocity_proof_step_6} and \eqref{Eqn:Backwards_reciprocity_proof_step_7}, Eq.~\eqref{Eqn:Backwards_reciprocity_proof_step_5} is equivalent to the following: 
\begin{align}
\label{Eqn:Backwards_reciprocity_proof_step_8}
&\left\{
\mathbf{u}_{1}^{(2)} \star \gamma \, \mathbf{b}_{1}^{(1)}
+ \frac{\gamma}{\phi_1(\mathbf{x})} 
\mathbf{u}_1^{(2)}(\mathbf{x},t) \bullet 
\mathbf{u}_{01}^{(1)}(\mathbf{x})
\right\}
- \left\{
\mathbf{u}_{1}^{(1)} \star \gamma \, \mathbf{b}_{1}^{(2)}
+ \frac{\gamma}{\phi_1(\mathbf{x})} 
\mathbf{u}_1^{(1)}(\mathbf{x},t) \bullet 
\mathbf{u}_{01}^{(2)}(\mathbf{x})
\right\} \nonumber \\ 
&\qquad \qquad \qquad = \mathbf{u}_{1}^{(2)} \star 
\mathrm{grad}\big[p^{(1)}_{1}\big]
- \mathbf{u}_{1}^{(1)} \star 
\mathrm{grad}\big[p^{(2)}_{1}\big]
\end{align}

Noting the terms in Eq.~\eqref{Eqn:Backwards_Reciprocal_relation}, we set out to simplify the following macro-related difference: 
\begin{align*}
    \Big\langle \mathcal{S}^{(2)}_{\mathrm{macro}};
    \mathcal{P}^{(1)}_{\mathrm{macro}}
    \Big\rangle_{\mathrm{macro}} 
    - \Big\langle \mathcal{S}^{(1)}_{\mathrm{macro}};
    \mathcal{P}^{(2)}_{\mathrm{macro}}
    \Big\rangle_{\mathrm{macro}} 
\end{align*}
Using Eq.~\eqref{Eqn:Backwards_macro_integrals}, we expand this  difference as follows: 
\begin{align}
    \label{Eqn:Backwards_reciprocity_proof_step_9}
    \Big\langle \mathcal{S}^{(2)}_{\mathrm{macro}};
    \mathcal{P}^{(1)}_{\mathrm{macro}}
    \Big\rangle_{\mathrm{macro}} 
    - \Big\langle \mathcal{S}^{(1)}_{\mathrm{macro}};
    \mathcal{P}^{(2)}_{\mathrm{macro}}
    \Big\rangle_{\mathrm{macro}} 
    &=  \int_{\Omega} \left\{
\mathbf{u}_{1}^{(2)} \star \gamma \, \mathbf{b}_{1}^{(1)}
+ \frac{\gamma}{\phi_1(\mathbf{x})} 
\mathbf{u}_1^{(2)}(\mathbf{x},t) \bullet 
\mathbf{u}_{01}^{(1)}(\mathbf{x})
\right\} \mathrm{d} \Omega \nonumber \\
&- \int_{\Omega} \left\{
\mathbf{u}_{1}^{(1)} \star \gamma \, \mathbf{b}_{1}^{(2)}
+ \frac{\gamma}{\phi_1(\mathbf{x})} 
\mathbf{u}_1^{(1)}(\mathbf{x},t) \bullet 
\mathbf{u}_{01}^{(2)}(\mathbf{x})
\right\} \mathrm{d} \Omega \nonumber \\
&-\int_{\Gamma^p_1}
\big(\mathbf{u}_{1}^{(2)} \bullet \widehat{\mathbf{n}}(\mathbf{x}) \big) 
\star p_{\mathrm{p}1}^{(1)}  \; \mathrm{d} \Gamma 
+\int_{\Gamma^u_1}
p_{1}^{(2)} \star u_{n1}^{(1)} \; \mathrm{d} \Gamma 
\nonumber \\
&+\int_{\Gamma^p_1}
\big(\mathbf{u}_{1}^{(1)} \bullet \widehat{\mathbf{n}}(\mathbf{x}) \big) 
\star p_{\mathrm{p}1}^{(2)}  \; \mathrm{d} \Gamma 
-\int_{\Gamma^u_1}
p_{1}^{(1)} \star u_{n1}^{(2)} \; \mathrm{d} \Gamma 
\end{align}
Equation~\eqref{Eqn:Backwards_reciprocity_proof_step_8} enables us to write the above equation as follows: 
\begin{align}
    \label{Eqn:Backwards_reciprocity_proof_step_10}
    \Big\langle \mathcal{S}^{(2)}_{\mathrm{macro}};
    \mathcal{P}^{(1)}_{\mathrm{macro}}
    \Big\rangle_{\mathrm{macro}} 
    - \Big\langle \mathcal{S}^{(1)}_{\mathrm{macro}};
    \mathcal{P}^{(2)}_{\mathrm{macro}}
    \Big\rangle_{\mathrm{macro}} 
    &= \int_{\Omega} \mathbf{u}_1^{(2)} \star 
    \mathrm{grad}\big[p_1^{(1)}\big] \, 
    \mathrm{d} \Omega 
    - \int_{\Omega} \mathbf{u}_1^{(1)} \star 
    \mathrm{grad}\big[p_1^{(2)}\big] \, 
    \mathrm{d} \Omega \nonumber \\
&-\int_{\Gamma^p_1}
\big(\mathbf{u}_{1}^{(2)} \bullet \widehat{\mathbf{n}}(\mathbf{x}) \big) 
\star p_{\mathrm{p}1}^{(1)}  \; \mathrm{d} \Gamma 
+\int_{\Gamma^u_1}
p_{1}^{(2)} \star u_{n1}^{(1)} \; \mathrm{d} \Gamma 
\nonumber \\
&+\int_{\Gamma^p_1}
\big(\mathbf{u}_{1}^{(1)} \bullet \widehat{\mathbf{n}}(\mathbf{x}) \big) 
\star p_{\mathrm{p}1}^{(2)}  \; \mathrm{d} \Gamma 
-\int_{\Gamma^u_1}
p_{1}^{(1)} \star u_{n1}^{(2)} \; \mathrm{d} \Gamma 
\end{align}
Applying Eq.~\eqref{Eqn:Backwards_Greens_identity} on the first two terms on the right side of the above equation, we write: 
\begin{align}
    \label{Eqn:Backwards_reciprocity_proof_step_11}
    \Big\langle \mathcal{S}^{(2)}_{\mathrm{macro}};
    \mathcal{P}^{(1)}_{\mathrm{macro}}
    \Big\rangle_{\mathrm{macro}} 
    - \Big\langle \mathcal{S}^{(1)}_{\mathrm{macro}};
    \mathcal{P}^{(2)}_{\mathrm{macro}}
    \Big\rangle_{\mathrm{macro}} 
    &= \int_{\partial \Omega} 
    \left(\mathbf{u}_{1}^{(2)} \bullet \widehat{\mathbf{n}}(\mathbf{x}) \right) 
    \star p_{1}^{(1)} \, \mathrm{d} \Gamma 
    - \int_{\Omega} 
    \mathrm{div}\big[\mathbf{u}_{1}^{(2)}\big] \star p_{1}^{(1)} \, \mathrm{d} \Omega 
    \nonumber \\
    &- \int_{\partial \Omega} 
    \left(\mathbf{u}_{1}^{(1)} \bullet 
    \widehat{\mathbf{n}}(\mathbf{x}) \right) 
    \star p_{1}^{(2)} \, 
    \mathrm{d} \Gamma  
    + \int_{\Omega} 
    \mathrm{div}\big[\mathbf{u}_{1}^{(1)}\big] 
    \star p_{1}^{(2)} \, 
    \mathrm{d} \Omega 
    \nonumber \\
    &-\int_{\Gamma^p_1}
    \big(\mathbf{u}_{1}^{(2)} \bullet \widehat{\mathbf{n}}(\mathbf{x}) \big) 
    \star p_{\mathrm{p}1}^{(1)}  \; \mathrm{d} \Gamma 
    +\int_{\Gamma^u_1}
    p_{1}^{(2)} \star u_{n1}^{(1)} \; \mathrm{d} \Gamma \nonumber \\
    &+\int_{\Gamma^p_1}
    \big(\mathbf{u}_{1}^{(1)} \bullet \widehat{\mathbf{n}}(\mathbf{x}) \big) 
    \star p_{\mathrm{p}1}^{(2)}  \; \mathrm{d} \Gamma 
    -\int_{\Gamma^u_1}
    p_{1}^{(1)} \star u_{n1}^{(2)} \; \mathrm{d} \Gamma 
\end{align}
Noting the partition of the boundary \eqref{Eqn:Backwards_macro_gamma_partition}, and invoking the boundary conditions for macro-network (i.e., Eqs.~\eqref{Eqn:Backwards_vBC_1} and \eqref{Eqn:Backwards_pBC_1}) and the commutative property of convolutions \eqref{Eqn:Backwards_commutative}, we simplify the above equation as follows: 
\begin{align}
    \label{Eqn:Backwards_reciprocity_proof_step_12}
    \Big\langle \mathcal{S}^{(2)}_{\mathrm{macro}};
    \mathcal{P}^{(1)}_{\mathrm{macro}}
    \Big\rangle_{\mathrm{macro}} 
    - \Big\langle \mathcal{S}^{(1)}_{\mathrm{macro}};
    \mathcal{P}^{(2)}_{\mathrm{macro}}
    \Big\rangle_{\mathrm{macro}} 
    = -\int_{\Omega} \mathrm{div}\big[\mathbf{u}_{1}^{(2)}\big] \star p_{1}^{(1)} \, \mathrm{d} \Omega 
    + \int_{\Omega} 
    \mathrm{div}\big[\mathbf{u}_{1}^{(1)}\big] \star 
    p^{(2)} \, \mathrm{d} \Omega 
\end{align}
The mass balance for the macro-network (i.e., Eq.~\eqref{Eqn:Backwards_BoM_1}) allows us to write the following: 
\begin{align}
    \label{Eqn:Backwards_reciprocity_proof_step_13}
    \Big\langle \mathcal{S}^{(2)}_{\mathrm{macro}};
    \mathcal{P}^{(1)}_{\mathrm{macro}}
    \Big\rangle_{\mathrm{macro}} 
    - \Big\langle \mathcal{S}^{(1)}_{\mathrm{macro}};
    \mathcal{P}^{(2)}_{\mathrm{macro}}
    \Big\rangle_{\mathrm{macro}} 
    &= \int_{\Omega} \frac{\beta}{\mu} \; 
    \Big(p^{(2)}_1 - p^{(2)}_2 \Big) \star 
    p^{(1)}_1 \; \mathrm{d} \Omega \nonumber \\
    &\qquad - \int_{\Omega} \frac{\beta}{\mu} \; 
    \Big(p^{(1)}_1 - p^{(1)}_2 \Big) \star 
    p^{(2)}_1 \; \mathrm{d} \Omega 
\end{align}

We now simplify the remaining terms in Eq.~\eqref{Eqn:Backwards_Reciprocal_relation}---the micro-related difference:
\begin{align*}
    \Big\langle \mathcal{S}^{(2)}_{\mathrm{micro}};
    \mathcal{P}^{(1)}_{\mathrm{micro}}
    \Big\rangle_{\mathrm{micro}} 
    - \Big\langle \mathcal{S}^{(1)}_{\mathrm{micro}};
    \mathcal{P}^{(2)}_{\mathrm{micro}}
    \Big\rangle_{\mathrm{micro}} 
\end{align*}
Following a similar procedure to the above (mimicking the steps in writing Eqs.~\eqref{Eqn:Backwards_reciprocity_proof_step_1}---\eqref{Eqn:Backwards_reciprocity_proof_step_13}), we arrive at the following expression for the micro-related difference: 
\begin{align}
    \label{Eqn:Backwards_reciprocity_proof_step_14}
    \Big\langle \mathcal{S}^{(2)}_{\mathrm{micro}};
    \mathcal{P}^{(1)}_{\mathrm{micro}}
    \Big\rangle_{\mathrm{micro}} 
    - \Big\langle \mathcal{S}^{(1)}_{\mathrm{micro}};
    \mathcal{P}^{(2)}_{\mathrm{micro}}
    \Big\rangle_{\mathrm{micro}} 
    &= -\int_{\Omega} \frac{\beta}{\mu} \; 
    \Big(p^{(2)}_1 - p^{(2)}_2 \Big) \star 
    p^{(1)}_2 \; \mathrm{d} \Omega \nonumber \\
    &\qquad + \int_{\Omega} \frac{\beta}{\mu} \; 
    \Big(p^{(1)}_1 - p^{(1)}_2 \Big) \star 
    p^{(2)}_2 \; \mathrm{d} \Omega 
\end{align}
The sign incongruity on the right sides of Eqs.~\eqref{Eqn:Backwards_reciprocity_proof_step_13} and \eqref{Eqn:Backwards_reciprocity_proof_step_14} is because of the sign disparity related to the mass transfer between the macro- and micro-networks (cf. Eqs.~\eqref{Eqn:Backwards_BoM_1} and \eqref{Eqn:Backwards_BoM_2}). 
Adding Eqs.~\eqref{Eqn:Backwards_reciprocity_proof_step_13} and \eqref{Eqn:Backwards_reciprocity_proof_step_14}, we infer the following:
\begin{align}
    \label{Eqn:Backwards_reciprocity_proof_step_15}
    &\Big\langle \mathcal{S}^{(2)}_{\mathrm{macro}};
    \mathcal{P}^{(1)}_{\mathrm{macro}}
    \Big\rangle_{\mathrm{macro}} 
    - \Big\langle \mathcal{S}^{(1)}_{\mathrm{macro}};
    \mathcal{P}^{(2)}_{\mathrm{macro}}
    \Big\rangle_{\mathrm{macro}} 
    +\Big\langle \mathcal{S}^{(2)}_{\mathrm{micro}};
    \mathcal{P}^{(1)}_{\mathrm{micro}}
    \Big\rangle_{\mathrm{micro}} 
    - \Big\langle \mathcal{S}^{(1)}_{\mathrm{micro}};
    \mathcal{P}^{(2)}_{\mathrm{micro}}
    \Big\rangle_{\mathrm{micro}} \nonumber \\
    &\qquad \qquad = \int_{\Omega} \frac{\beta}{\mu} \; 
    \Big(p^{(2)}_1 - p^{(2)}_2 \Big) \star 
    \Big(p^{(1)}_1 - p^{(1)}_2\Big) \; \mathrm{d} \Omega 
    - \int_{\Omega} \frac{\beta}{\mu} \; 
    \Big(p^{(1)}_1 - p^{(1)}_2 \Big) \star 
    \Big(p^{(2)}_1 - p^{(2)}_2\Big) \; \mathrm{d} \Omega \nonumber \\ 
    & \qquad \qquad = 0
\end{align}
Rearranging the terms, we have established the desired relationship: 
\begin{align}
    \label{Eqn:Backwards_reciprocity_proof_step_16}
    \Big\langle \mathcal{S}^{(2)}_{\mathrm{macro}};
    \mathcal{P}^{(1)}_{\mathrm{macro}}
    \Big\rangle_{\mathrm{macro}} 
    + \Big\langle \mathcal{S}^{(2)}_{\mathrm{micro}};
    \mathcal{P}^{(1)}_{\mathrm{micro}}
    \Big\rangle_{\mathrm{micro}} 
    = \Big\langle \mathcal{S}^{(1)}_{\mathrm{macro}};
    \mathcal{P}^{(2)}_{\mathrm{macro}}
    \Big\rangle_{\mathrm{macro}} 
    + \Big\langle \mathcal{S}^{(1)}_{\mathrm{micro}};
    \mathcal{P}^{(2)}_{\mathrm{micro}}
    \Big\rangle_{\mathrm{micro}} 
\end{align}
\end{proof}

\section{A VARIATIONAL PRINCIPLE}
\label{Sec:S5_Backwards_Variational}

Although the DPP model has first-order time derivatives, the model still enjoys a variational principle, as shown below. We first establish a useful intermediate result. 

\begin{lemma}
\label{Eqn:Lemma_variational_intermediate_result}
A solution of the original governing equations \eqref{Eqn:Backwards_BoLM_1}--\eqref{Eqn:Backwards_vIC_2} is also a solution of the following and \emph{vice versa}: 
\begin{subequations}
\begin{align}
    \label{Eqn:Backwards_equivalent_BoLM_1}
    &\frac{\gamma}{\phi_1(\mathbf{x})} \mathbf{u}_1(\mathbf{x},t) 
    + 1 \star \mu \, \mathbf{K}_{1}^{-1}(\mathbf{x}) \, \mathbf{u}_{1} + 1 \star \mathrm{grad}[p_{1}] = 1 \star \gamma \, \mathbf{b}_{1} + \frac{\gamma}{\phi_{1}(\mathbf{x})} \mathbf{u}_{01}(\mathbf{x}) \\ 
    \label{Eqn:Backwards_equivalent_BoLM_2}
    &\frac{\gamma}{\phi_2(\mathbf{x})} \mathbf{u}_2(\mathbf{x},t)  
    + 1 \star \mu \, \mathbf{K}_{2}^{-1}(\mathbf{x}) \, \mathbf{u}_{2} + 1 \star \mathrm{grad}[p_{2}] 
    = 1 \star \gamma \, \mathbf{b}_{2} 
    + \frac{\gamma}{\phi_{2}(\mathbf{x})} \mathbf{u}_{02}(\mathbf{x}) \\
    \label{Eqn:Backwards_equivalent_BoM_1}
    &1 \star \mathrm{div}[\mathbf{u}_1] = -1 \star \frac{\beta}{\mu} \, \big(p_1 - p_2\big)  \\
    \label{Eqn:Backwards_equivalent_BoM_2}
    &1 \star \mathrm{div}[\mathbf{u}_2] = +1 \star \frac{\beta}{\mu} \, \big(p_1 - p_2\big) \\
    & 1 \star \mathbf{u}_1 \bullet \widehat{\mathbf{n}}(\mathbf{x}) = 1 \star u_{n1} \\
    & 1 \star \mathbf{u}_2 \bullet 
    \widehat{\mathbf{n}}(\mathbf{x}) = 1 \star u_{n2} \\ 
    & 1 \star p_1 = 1 \star p_{\mathrm{p}1} \\
    \label{Eqn:Backwards_equivalent_pBC_2}
    & 1 \star p_2 = 1 \star p_{\mathrm{p}2}
\end{align}
\end{subequations}
\end{lemma}
\begin{proof}
    We first show that Eqs.~\eqref{Eqn:Backwards_equivalent_BoLM_1}--\eqref{Eqn:Backwards_equivalent_pBC_2} imply Eqs.~\eqref{Eqn:Backwards_BoLM_1}--\eqref{Eqn:Backwards_vIC_2}.

    To establish the initial velocity conditions (i.e., Eqs.~\eqref{Eqn:Backwards_vIC_1} and \eqref{Eqn:Backwards_vIC_2}), we apply limit $t \rightarrow 0$ on both sides of Eq.~\eqref{Eqn:Backwards_equivalent_BoLM_1}:
    \begin{align}
    \label{Eqn:Backwards_Lemma_step_1}
    \lim_{t \rightarrow 0} \; \Big\{\frac{\gamma}{\phi_1(\mathbf{x})} \mathbf{u}_1(\mathbf{x},t) 
    + 1 \star \mu \, \mathbf{K}_{1}^{-1}(\mathbf{x}) \, \mathbf{u}_{1} + 1 \star \mathrm{grad}[p_{1}] 
    \Big\} 
    = \lim_{t \rightarrow 0} \Big\{1 \star \gamma \, \mathbf{b}_{1} + \frac{\gamma}{\phi_{1}(\mathbf{x})} \mathbf{u}_{01}(\mathbf{x}) 
    \Big\} 
    \end{align}
    The limiting process applied to the first term in the above equation gives rise to the following: 
    \begin{align}
    \label{Eqn:Backwards_Lemma_step_2}
    \lim_{t \rightarrow 0} \; 
    \frac{\gamma}{\phi_1(\mathbf{x})} \mathbf{u}_1(\mathbf{x},t) 
    = \frac{\gamma}{\phi_1(\mathbf{x})} 
    \lim_{t \rightarrow 0} \; \mathbf{u}_1(\mathbf{x},t) 
    = \frac{\gamma}{\phi_1(\mathbf{x})} 
    \; \mathbf{u}_{1}(\mathbf{x},0) 
    \end{align}
    For the second term, the limiting process amounts to the following: 
    \begin{align}
    \lim_{t \rightarrow 0} \; 
    1 \star \mu \, \mathbf{K}_{1}^{-1}(\mathbf{x}) \, \mathbf{u}_1 
    = \lim_{t \rightarrow 0} \; 
    \int_{0}^{t} \mu \, \mathbf{K}_{1}^{-1}(\mathbf{x}) \, \mathbf{u}_1(\mathbf{x},\tau) \; \mathrm{d} \tau  
    = 0
    \end{align}
    Similarly, the third and fourth terms vanish: 
    \begin{align}
    \lim_{t \rightarrow 0} \; 
    1 \star \mathrm{grad}[p_1] = 0
    \quad \mathrm{and} \quad 
    \lim_{t \rightarrow 0} \; 
    1 \star \gamma \, \mathbf{b}_1 = 0
    \end{align}
    Since the last term is independent of time, we have:  
    \begin{align}
    \label{Eqn:Backwards_Lemma_step_5}
    \lim_{t \rightarrow 0} \; \frac{\gamma}{\phi_1(\mathbf{x})} \, \mathbf{u}_{01}(\mathbf{x}) 
    = \frac{\gamma}{\phi_1(\mathbf{x})} \, \mathbf{u}_{01}(\mathbf{x}) 
    \end{align}
    By virtue of the above four equations \eqref{Eqn:Backwards_Lemma_step_2}--\eqref{Eqn:Backwards_Lemma_step_5}, Eq.~\eqref{Eqn:Backwards_Lemma_step_1} verifies the first initial velocity condition \eqref{Eqn:Backwards_vIC_1}:
    \begin{align}
    \frac{\gamma}{\phi_1(\mathbf{x})} \, \mathbf{u}_{1}(\mathbf{x},0) 
    = \frac{\gamma}{\phi_1(\mathbf{x})} \, \mathbf{u}_{01}(\mathbf{x}) 
    \; \implies \; 
    \mathbf{u}_{1}(\mathbf{x},0) 
    = \mathbf{u}_{01}(\mathbf{x}) 
    \end{align}
    Following a similar procedure, one can verify the other initial velocity condition \eqref{Eqn:Backwards_vIC_2}.

    Applying time derivative and using the property of convolutions given by Eq.~\eqref{Eqn:Backwards_convolutions_Leibniz_rule} (which is based on Leibniz integral rule) verifies the rest of the equations. We illustrate this procedure on the first equation \eqref{Eqn:Backwards_equivalent_BoLM_1} by taking the time derivative on both sides of the equation: 
    \begin{align}
    \label{Eqn:Backwards_Lemma_step_7}
    \frac{\partial}{\partial t} 
    \left\{\frac{\gamma}{\phi_1(\mathbf{x})} \mathbf{u}_1(\mathbf{x},t) 
    + 1 \star \mu \, \mathbf{K}_{1}^{-1}(\mathbf{x}) \, \mathbf{u}_{1} + 1 \star \mathrm{grad}[p_{1}] \right\} 
    = \frac{\partial}{\partial t} 
    \left\{1 \star \gamma \, \mathbf{b}_{1} + \frac{\gamma}{\phi_{1}(\mathbf{x})} \mathbf{u}_{01}(\mathbf{x}) \right\} 
    \end{align}
    The first term can be written as follows: 
    \begin{align}
        \label{Eqn:Backwards_Lemma_step_8}
        \frac{\partial}{\partial t}\left\{\frac{\gamma}{\phi_1(\mathbf{x})} \mathbf{u}_1(\mathbf{x},t)\right\} 
        = \frac{\gamma}{\phi_1(\mathbf{x})} 
        \frac{\partial \mathbf{u}_1(\mathbf{x},t)}{\partial t}
    \end{align}
    Invoking the property \eqref{Eqn:Backwards_convolutions_Leibniz_rule}, the second, third, and fourth terms can be written as follows:
    \begin{align}
    \label{Eqn:Backwards_Lemma_step_9_a}
    &\frac{\partial  (1 \star \mu \, \mathbf{K}_{1}^{-1}(\mathbf{x}) \, \mathbf{u}_1)}{\partial t}
    = \mu \, \mathbf{K}_{1}^{-1}(\mathbf{x}) \, \mathbf{u}_1(\mathbf{x},t) \\ 
    \label{Eqn:Backwards_Lemma_step_9_b}
    &\frac{\partial \big(1 \star \mathrm{grad}[p_1]\big)}{\partial t} 
    = \mathrm{grad}\big[p_1(\mathbf{x},t)\big] \\
    \label{Eqn:Backwards_Lemma_step_9_c}
    &\frac{\partial \big(1 \star \gamma \, \mathbf{b}_1\big)}{\partial t} 
    = \gamma \, \mathbf{b}_{1}(\mathbf{x},t)
    \end{align}
    The last term is independent of time, and we thus have: 
    \begin{align}
    \label{Eqn:Backwards_Lemma_step_10}
    &\frac{\partial}{\partial t} 
    \left\{\frac{\gamma}{\phi_1(\mathbf{x})} \, \mathbf{u}_{01}(\mathbf{x}) \right\} 
    = \mathbf{0}
    \end{align}
    Eqs.~\eqref{Eqn:Backwards_Lemma_step_8}--\eqref{Eqn:Backwards_Lemma_step_10} alongside Eq.~\eqref{Eqn:Backwards_Lemma_step_7} verifies Eq.~\eqref{Eqn:Backwards_BoLM_1}. Following a similar procedure, we can verify the other governing equations \eqref{Eqn:Backwards_BoLM_2}--\eqref{Eqn:Backwards_pBC_2}. 

    We now show that Eqs.~\eqref{Eqn:Backwards_BoLM_1}--\eqref{Eqn:Backwards_vIC_2} imply Eqs.~\eqref{Eqn:Backwards_equivalent_BoLM_1}--\eqref{Eqn:Backwards_equivalent_pBC_2}. Equations~\eqref{Eqn:Backwards_equivalent_BoM_1}--\eqref{Eqn:Backwards_equivalent_pBC_2} are, respectively, mere convolutions of Eqs.~\eqref{Eqn:Backwards_BoM_1}--\eqref{Eqn:Backwards_pBC_2} with ``$1$.'' 
    Equations~\eqref{Eqn:Backwards_equivalent_BoLM_1} and \eqref{Eqn:Backwards_equivalent_BoLM_2} can be obtained by convolving Eqs.~\eqref{Eqn:Backwards_BoLM_1} and \eqref{Eqn:Backwards_BoLM_2} with unity and utilizing the initial velocity conditions \eqref{Eqn:Backwards_vBC_1} and \eqref{Eqn:Backwards_vBC_2}, respectively. This completes the proof of the lemma. 
\end{proof}

\begin{theorem}[Variational principle]
Consider the following functional: 
\begin{align}
    \Psi[\mathbf{u}_1,\mathbf{u}_2,p_1,p_2] 
    &:= \int_{\Omega} \frac{1}{2} \, \frac{\gamma}{\phi_1(\mathbf{x})} \mathbf{u}_1  \star 
    \mathbf{u}_1 \; \mathrm{d} \Omega 
    + \int_{\Omega} \frac{1}{2} \, \frac{\gamma}{\phi_2(\mathbf{x})} \mathbf{u}_2  \star 
    \mathbf{u}_2 \; \mathrm{d} \Omega \nonumber \\
    &\qquad + \int_{\Omega} \frac{1}{2} \, \mathbf{u}_1 \star 1 \star \mu \, \mathbf{K}_{1}^{-1}(\mathbf{x}) \, \mathbf{u}_{1} \; \mathrm{d} \Omega 
    + \int_{\Omega} \frac{1}{2} \, \mathbf{u}_2 \star 1 \star \mu \, \mathbf{K}_{2}^{-1}(\mathbf{x}) \, \mathbf{u}_{2} \; \mathrm{d} \Omega \nonumber \\ 
    &\qquad - \int_{\Omega} \mathrm{div}[\mathbf{u}_1] \star 1 \star p_{1} \; \mathrm{d} \Omega  
    - \int_{\Omega} \mathrm{div}[\mathbf{u}_2] \star 1 \star p_{2} \; \mathrm{d} \Omega  \nonumber \\
    &\qquad - \int_{\Omega} \frac{\beta}{2 \, \mu}
    \big(p_1 - p_2\big) \star 1 \star \big(p_{1} - p_2\big) \; \mathrm{d} \Omega  \nonumber \\
    &\qquad  + \int_{\Gamma_1^{u}} u_{n1} \star 1 \star p_1 \; \mathrm{d} \Gamma 
    + \int_{\Gamma_1^{p}} \mathbf{u}_1 \bullet \widehat{\mathbf{n}}(\mathbf{x}) \star 1 \star p_{\mathrm{p}1} \; 
    \mathrm{d} \Gamma  \nonumber \\
    &\qquad  + \int_{\Gamma_2^{u}} u_{n2} \star 1 \star p_2 \; \mathrm{d} \Gamma
    + \int_{\Gamma_2^{p}} \mathbf{u}_2 \bullet \widehat{\mathbf{n}}(\mathbf{x}) \star 1 \star p_{\mathrm{p}2} \; 
    \mathrm{d} \Gamma  \nonumber \\
    &\qquad - \int_{\Omega} \mathbf{u}_{1} \star 1 \star \gamma \, \mathbf{b}_{1} \; \mathrm{d} \Omega 
    - \int_{\Omega} \frac{\gamma}{\phi_{1}(\mathbf{x})} \mathbf{u}_1 \star \mathbf{u}_{01}(\mathbf{x}) \;  \mathrm{d} \Omega \nonumber \\
    &\qquad - \int_{\Omega} \mathbf{u}_{2} \star 1 \star \gamma \, \mathbf{b}_{2} \; \mathrm{d} \Omega 
    - \int_{\Omega} \frac{\gamma}{\phi_{2}(\mathbf{x})} \mathbf{u}_2 \star \mathbf{u}_{02}(\mathbf{x}) \;  \mathrm{d} \Omega 
\end{align}
The corresponding G\^ateaux variation is defined as follows: 
\begin{align}
    \delta \Psi\big[\mathbf{u}_1,\mathbf{u}_2,p_1,p_2;\delta \mathbf{u}_1,\delta \mathbf{u}_2,\delta p_1,\delta p_2\big] 
    := \left[\frac{d \, \Psi\big[\mathbf{u}_1 + \epsilon \, \delta \mathbf{u}_1,\mathbf{u}_2 + \epsilon \, \delta \mathbf{u}_2, p_1 + \epsilon \, \delta p_1,p_2 + \epsilon \, \delta p_2\big]}{d \epsilon} \right]_{\epsilon = 0}
\end{align}
Then the solution of the following equation is a solution of the original governing equations \eqref{Eqn:Backwards_BoLM_1}--\eqref{Eqn:Backwards_vIC_2}: 
\begin{align}
    \delta \Psi\big[\mathbf{u}_1,\mathbf{u}_2,p_1,p_2;\delta \mathbf{u}_1,\delta \mathbf{u}_2,\delta p_1,\delta p_2\big] 
    = 0 \quad \forall \delta \mathbf{u}_1,\delta \mathbf{u}_2,\delta p_1,\delta p_2
\end{align}
\end{theorem}
\begin{proof}
    We first calculate the G\^ateaux variation $\delta \Psi$: 
    \begin{align}
    \delta \Psi[\mathbf{u}_1,\mathbf{u}_2,p_1,p_2;\delta\mathbf{u}_1,\delta\mathbf{u}_2,\delta p_1,\delta p_2] 
    &:= \int_{\Omega} \frac{\gamma}{\phi_1(\mathbf{x})} \delta \mathbf{u}_1  \star 
    \mathbf{u}_1 \; \mathrm{d} \Omega 
    + \int_{\Omega} \frac{\gamma}{\phi_2(\mathbf{x})} \delta \mathbf{u}_2  \star 
    \mathbf{u}_2 \; \mathrm{d} \Omega \nonumber \\
    &+ \int_{\Omega} \delta \mathbf{u}_1 \star 1 \star \mu \, \mathbf{K}_{1}^{-1}(\mathbf{x}) \, \mathbf{u}_{1} \; \mathrm{d} \Omega 
    + \int_{\Omega} \delta \mathbf{u}_2 \star 1 \star \mu \, \mathbf{K}_{2}^{-1}(\mathbf{x}) \, \mathbf{u}_{2} \; \mathrm{d} \Omega \nonumber \\ 
    &- \int_{\Omega} \mathrm{div}[\delta \mathbf{u}_1] \star 1 \star p_{1} \; \mathrm{d} \Omega  
    - \int_{\Omega} \mathrm{div}[\mathbf{u}_1] \star 1 \star \delta p_{1} \; \mathrm{d} \Omega  \nonumber \\
    &- \int_{\Omega} \mathrm{div}[\delta \mathbf{u}_2] \star 1 \star p_{2} \; \mathrm{d} \Omega 
    - \int_{\Omega} \mathrm{div}[\mathbf{u}_2] \star 1 \star \delta p_{2} \; \mathrm{d} \Omega
    \nonumber \\
    &- \int_{\Omega} \frac{\beta}{\mu}
    \big(\delta p_1 - \delta p_2\big) \star 1 \star \big(p_{1} - p_2\big) \; \mathrm{d} \Omega  \nonumber \\
    &+ \int_{\Gamma_1^{u}} u_{n1} \star 1 \star \delta p_1 \; \mathrm{d} \Gamma 
    + \int_{\Gamma_1^{p}} \delta \mathbf{u}_1 \bullet \widehat{\mathbf{n}}(\mathbf{x}) \star 1 \star p_{\mathrm{p}1} \; 
    \mathrm{d} \Gamma  \nonumber \\
    &+ \int_{\Gamma_2^{u}} u_{n2} \star 1 \star \delta p_2 \; \mathrm{d} \Gamma
    + \int_{\Gamma_2^{p}} \delta \mathbf{u}_2 \bullet \widehat{\mathbf{n}}(\mathbf{x}) \star 1 \star p_{\mathrm{p}2} \; 
    \mathrm{d} \Gamma  \nonumber \\
    &- \int_{\Omega} \delta \mathbf{u}_{1} \star 1 \star \gamma \, \mathbf{b}_{1} \; \mathrm{d} \Omega 
    - \int_{\Omega} \frac{\gamma}{\phi_{1}(\mathbf{x})} \delta \mathbf{u}_1 \star \mathbf{u}_{01}(\mathbf{x}) \;  \mathrm{d} \Omega \nonumber \\
    &- \int_{\Omega} \delta \mathbf{u}_{2} \star 1 \star \gamma \, \mathbf{b}_{2} \; \mathrm{d} \Omega 
    - \int_{\Omega} \frac{\gamma}{\phi_{2}(\mathbf{x})} \delta \mathbf{u}_2 \star \mathbf{u}_{02}(\mathbf{x}) \;  \mathrm{d} \Omega 
\end{align}
Using Green's identity and the symmetry property of convolutions, we arrange the terms as follows: 
\begin{align}
    &\delta \Psi[\mathbf{u}_1,\mathbf{u}_2,p_1,p_2;\delta \mathbf{u}_1,\delta \mathbf{u}_2,\delta p_1,\delta p_2] \nonumber \\ 
    &\quad = \int_{\Omega} \delta \mathbf{u}_1 \star 
    \left(\frac{\gamma}{\phi_1(\mathbf{x})} 
    \mathbf{u}_1 
    + 1 \star \mu \, \mathbf{K}_{1}^{-1}(\mathbf{x}) \, \mathbf{u}_{1} 
    + 1 \star \mathrm{grad}[p_1] - 1 \star \gamma \, \mathbf{b}_{1} - \frac{\gamma}{\phi_{1}(\mathbf{x})} \mathbf{u}_{01}(\mathbf{x})  \right) \; \mathrm{d} \Omega \nonumber \\ 
    &\qquad + \int_{\Omega} \delta \mathbf{u}_2 \star 
    \left(\frac{\gamma}{\phi_2(\mathbf{x})} 
    \mathbf{u}_2 
    + 1 \star \mu \, \mathbf{K}_{2}^{-1}(\mathbf{x}) \, \mathbf{u}_{2} 
    + 1 \star \mathrm{grad}[p_2] - 1 \star \gamma \, \mathbf{b}_{2} - \frac{\gamma}{\phi_{2}(\mathbf{x})} \mathbf{u}_{02}(\mathbf{x})  \right) \; \mathrm{d} \Omega \nonumber \\ 
    &\qquad - \int_{\Omega} \delta p_1 \star 
    1 \star \left(\mathrm{div}[\mathbf{u}_1] 
    + \frac{\beta}{\mu}
    \big(p_{1} - p_2\big) \right) 
    \; \mathrm{d} \Omega  
    - \int_{\Omega} \delta p_2 \star 
    1 \star  
    \left(\mathrm{div}[\mathbf{u}_2] 
    - \frac{\beta}{\mu}
    \big(p_{1} - p_2\big) \right) 
    \; \mathrm{d} \Omega  \nonumber \\
    &\qquad  - \int_{\Gamma_1^{u}}
    \delta p_1 \star 1 \star 
    \big(\mathbf{u}_1 \bullet \widehat{\mathbf{n}}(\mathbf{x}) - u_{n1}\big) \; \mathrm{d} \Gamma
    - \int_{\Gamma_2^{u}}
    \delta p_2 \star 1 \star 
    \big(\mathbf{u}_2 \bullet \widehat{\mathbf{n}}(\mathbf{x}) - u_{n2}\big) \; \mathrm{d} \Gamma 
    \nonumber \\ 
    &\qquad - \int_{\Gamma_1^{p}} \delta \mathbf{u}_1 \bullet \widehat{\mathbf{n}}(\mathbf{x}) \star 1 \star \big(p_1 - p_{\mathrm{p}1}\big) \; 
    \mathrm{d} \Gamma 
    - \int_{\Gamma_2^{p}} \delta \mathbf{u}_2 \bullet \widehat{\mathbf{n}}(\mathbf{x}) \star 1 \star \big(p_2 - p_{\mathrm{p}2}\big) \; 
    \mathrm{d} \Gamma   
\end{align}

Now invoking the condition that $\delta \Psi$ vanishes for any arbitrary choice of $\{\delta \mathbf{u}_1,\delta \mathbf{u}_2,\delta p_1,\delta p_2\}$, the fundamental theorem of calculus implies that Eqs.~\eqref{Eqn:Backwards_equivalent_BoLM_1}--\eqref{Eqn:Backwards_equivalent_pBC_2} are met. 
Lemma \ref{Eqn:Lemma_variational_intermediate_result} implies that the solution of Eqs.~\eqref{Eqn:Backwards_equivalent_BoLM_1}--\eqref{Eqn:Backwards_equivalent_pBC_2} is a solution of the original governing equations \eqref{Eqn:Backwards_BoLM_1}--\eqref{Eqn:Backwards_vIC_2}.
\end{proof}

\section{CLOSURE}
\label{Sec:S6_Backwards_Closure}

We presented three qualitative properties the double porosity/permeability (DPP) model satisfies in the transient regime: backward-in-time uniqueness of solutions, reciprocal theorem, similar to Betti's reciprocal relations available in linearized elasticity, and a variational principle. These results add to the repertoire of the results available in the transient regime under the DPP model. All these results advance the theoretical understanding of the DPP model and serve as \emph{a posteriori} error measures to assess the accuracy of numerical formulations. The ``backward-in-time uniqueness'' property puts the inverse problem of identifying the initial condition on a firm footing. Two possible future works are: (1) utilizing these qualitative properties in numerical verification studies and (2) developing an inversion framework for transient DPP models.

\section*{DATA AVAILABILITY}
Data sharing does not apply to this article as no new data were created or analyzed.

\bibliographystyle{plainnat}
\bibliography{Master_References}
\end{document}